\documentclass[11pt, a4paper]{article}
\usepackage[margin=0.75in]{geometry}
\usepackage{caption}

\usepackage{bm}
\usepackage{extarrows}
\usepackage{stmaryrd}
\usepackage{amsmath}
\usepackage{amssymb}
\usepackage{amsthm}
\usepackage{amsfonts}
\usepackage{algorithm}
\usepackage{algpseudocode}
\usepackage{graphicx}
\usepackage{color}
\usepackage{relsize}
\usepackage[toc,page]{appendix}
\usepackage{hyperref}
\usepackage{cleveref}
\usepackage[font=small,labelfont=bf]{caption}

\newcommand{\lbm}{lattice Boltzmann~}
\newcommand{\fd}{Finite Difference~}
\newcommand{\spatialdimensionality}{d}
\newcommand{\spacevariable}{x}
\newcommand{\freqvariable}{\xi}
\newcommand{\timevariable}{n}
\newcommand{\latticevelocity}{\lambda}
\newcommand{\spacestep}{\Delta x}
\newcommand{\timestep}{\Delta t}
\newcommand{\matricial}[1]{\bm{#1}}
\newcommand{\vectorial}[1]{\bm{#1}}
\newcommand{\reals}{\mathbb{R}}
\newcommand{\naturals}{\mathbb{N}}
\newcommand{\relatives}{\mathbb{Z}}
\newcommand{\complex}{\mathbb{C}}

\newcommand{\definitionequality}{:=}
\newcommand{\velocityletter}{e}
\newcommand{\normalizedvelocityletter}{c}
\newcommand{\velocitynumber}{q}
\newcommand{\populationindex}{j}
\newcommand{\shiftoperator}[1]{T_{\spacestep}^{#1}}
\newcommand{\lattice}{\mathcal{L}}
\newcommand{\fourier}[1]{\hat{#1}}

\newcommand{\imag}{\imath}%
\newcommand{\shiftoperatorfourier}[1]{\fourier{T}\vphantom{T}^{#1}_{\spacestep}}
\newcommand{\momentsmatrix}{\matricial{M}}
\newcommand{\lineargroup}{\text{GL}}
\newcommand{\distributionletter}{f}
\newcommand{\momentletter}{m}
\newcommand{\collided}{\star}
\newcommand{\identity}{I}
\newcommand{\relaxationmatrix}{\matricial{S}}
\newcommand{\atequilibrium}{\text{eq}}
\newcommand{\matrixspace}{\mathcal{M}}
\newcommand{\rank}{\text{rank}}
\newcommand{\consmomentsnumber}{N}
\newcommand{\diagmatrix}{\text{diag}}
\newcommand{\relaxparletter}{s}
\newcommand{\nontrivialnumber}{Q}
\newcommand{\streammoments}{\matricial{T}}
\newcommand{\schememoments}{\matricial{A}}
\newcommand{\schemeequil}{\matricial{B}}

\newcommand{\setfinitedifferenceoperators}{\mathcal{D}^{\spatialdimensionality}_{\spacestep}}
\newcommand{\setfinitedifferenceoperatorsfourier}{\fourier{\mathcal{D}}\vphantom{\mathcal{D}}^{\spatialdimensionality}_{\spacestep}}
\newcommand{\sumoperators}{+}
\newcommand{\prodoperators}{\circ}
\newcommand{\basicx}{\mathsf{x}}
\newcommand{\basicy}{\mathsf{y}}
\newcommand{\basicz}{\mathsf{z}}
\newcommand{\determinant}{\text{det}}
\newcommand{\charactpolynomial}{\mathlarger{\chi}}
\newcommand{\minimalpolynomial}{\mu}
\newcommand{\minimalannpolynomial}{\nu}
\newcommand{\minimalannrowpolynomial}{\tilde{\minimalannpolynomial}}
\newcommand{\polynomialunknown}{X}
\newcommand{\degree}{\text{deg}}
\newcommand{\scheme}[2]{$\text{D}_{#1}\text{Q}_{#2}$}
\newcommand{\conj}[1]{\overline{#1}}
\newcommand{\avglink}{\mathsf{A}}
\newcommand{\difflink}{\mathsf{D}}
\newcommand{\coeffminimal}{\omega}
\newcommand{\coeffcharact}{\gamma}
\newcommand{\coeffannminimal}{\psi}
\newcommand{\matrixminannpol}{\matricial{V}}
\newcommand{\kernel}{\text{ker}}

\newcommand{\indicemoments}{i}
\newcommand{\indiceconserved}{\indicemoments}
\newcommand{\schememomentsother}{\matricial{A}^{\diamond}}

\newcommand{\bigO}[1]{\mathcal{O}(#1)}
\newcommand{\integerinterval}[2]{\llbracket #1, #2 \rrbracket}
\newcommand{\averageaxis}{\mathsf{A}_{\text{a}}}
\newcommand{\averagediagonal}{\mathsf{A}_{\text{d}}}
\newcommand{\genericfunction}{f}
\newcommand{\coeffgenericfinitedifference}{\alpha}
\newcommand{\transpose}[1]{#1^{\intercal}}
\newcommand{\odevariable}{y}
\newcommand{\odeinitialdatum}{\hat{\odevariable}}
\newcommand{\numbercoupleslinkscheme}{W}

\newcommand{\linearequilvector}{\vectorial{\epsilon}}
\newcommand{\canonicalbasisvector}{\vectorial{e}}
\newcommand{\adjugate}{\text{adj}}
\newcommand{\genericcommutativering}{\mathcal{R}}
\newcommand{\ellpspace}[1]{L^{#1}}
\newcommand{\elltwospace}{\ellpspace{2}}

\newcommand{\amplificationpolynomial}{\Phi}
\newcommand{\coeffamplificationpoly}{\varphi}
\newcommand{\rootamplificationpoly}{g}
\newcommand{\groupunitsshiftoperators}{\mathcal{T}^{\spatialdimensionality}_{\spacestep}}
\newcommand{\groupunitsshiftoperatorsfourier}{\fourier{\mathcal{T}}\vphantom{\mathcal{T}}^{\spatialdimensionality}_{\spacestep}}
\newcommand{\genericunit}{\mathsf{T}}
\newcommand{\genericfinitedifference}{\mathsf{D}}
\newcommand{\generatorset}[1]{\langle #1 \rangle}

\newcommand{\setoflines}{I}

\newcommand{\indiceslines}{i}
\newcommand{\indicescolumns}{j}
\newcommand{\genericmatrix}{\matricial{C}}
\newcommand{\genericmatrixtwo}{\matricial{D}}
\newcommand{\setcardinality}[1]{|#1|}
\newcommand{\indicetimeshift}{k}
\newcommand{\indicepolynomials}{\indicetimeshift}

\newcommand{\oforder}{\sim}
\newcommand{\thup}[1]{#1$^{\text{th}}$}
\newcommand{\cutmatrixsquare}[2]{#1_{#2}}
\newcommand{\cutmatrixsquaretrimmed}[2]{#1[#2]}
\newcommand{\fouriertransformletter}{\mathcal{F}}
\newcommand{\fouriertransform}[1]{\fouriertransformletter[#1]}
\newcommand{\spectrum}[1]{\sigma(#1)}
\newcommand{\cfl}{\mathtt{C}}
\newcommand{\foueriernumber}{\mathtt{D}}
\newcommand{\evaluationOperator}{\mathcal{E}}
\newcommand{\coscomp}{\gamma}
\newcommand{\fdvariable}{u}

\newcommand{\evaluationEquilibria}[1]{\hspace{-0.1cm} \mid^{#1}}

\providecommand{\keywords}[1]{\textbf{\textit{Keywords ---}} #1}
\providecommand{\pacs}[1]{\textbf{\textit{MSC Classification ---}} #1}

\newtheorem{theorem}{Theorem}
\newtheorem{proposition}[theorem]{Proposition}%
\newtheorem{corollary}[theorem]{Corollary}%
\newtheorem{lemma}[theorem]{Lemma}%

\newtheorem{example}{Example}%
\newtheorem{remark}{Remark}%

\newtheorem{definition}{Definition}%

\begin{document}

\title{Finite Difference formulation of any lattice Boltzmann scheme}
\author{Thomas Bellotti (\href{mailto: thomas.bellotti@polytechnique.edu}{thomas.bellotti@polytechnique.edu}) \\ {\footnotesize CMAP, CNRS, \'Ecole polytechnique, Institut Polytechnique de Paris, 91128 Palaiseau Cedex, France} \\
Benjamin Graille (\href{mailto: benjamin.graille@universite-paris-saclay.fr}{benjamin.graille@universite-paris-saclay.fr})  \\ {\footnotesize Institut de Math\'ematique d'Orsay, Universit\'e Paris-Saclay, 91405 Orsay Cedex, France} \\
Marc Massot (\href{mailto: marc.massot@polytechnique.edu}{marc.massot@polytechnique.edu}) \\ {\footnotesize CMAP, CNRS, \'Ecole polytechnique, Institut Polytechnique de Paris, 91128 Palaiseau Cedex, France} \\}

\date{}




\maketitle
\abstract{Lattice Boltzmann schemes rely on the enlargement of the size of the target problem in order to solve PDEs in a highly parallelizable and efficient kinetic-like fashion, split into a collision and a stream phase.
This structure, despite the well-known advantages from a computational standpoint, is not suitable to construct a rigorous notion of consistency with respect to the target equations and to provide a precise notion of stability.
In order to alleviate these shortages and introduce a rigorous framework, we demonstrate that any \lbm scheme can be rewritten as a corresponding multi-step \fd scheme on the conserved variables. 
This is achieved by devising a suitable formalism based on operators, commutative algebra and polynomials.
Therefore, the notion of consistency of the corresponding \fd scheme allows to invoke the Lax-Richtmyer theorem in the case of linear \lbm schemes.
Moreover, we show that the frequently-used von Neumann-like stability analysis for \lbm schemes entirely corresponds to the von Neumann stability analysis of their \fd counterpart.
More generally, the usual tools for the analysis of \fd schemes are now readily available to study \lbm schemes.
Their relevance is verified by means of numerical illustrations.} \vspace{0.2cm} \\ 
\keywords{Lattice Boltzmann methods, Finite Difference multi-step methods, consistency, von Neumann stability analysis, Cayley-Hamilton theorem on the ring of Finite Difference operators}  \vspace{0.2cm} \\ 
\pacs{76M28, 65M06, 65M12, 15A15}

\section{Introduction}

Lattice Boltzmann schemes are a class of computational methods used to simulate systems of conservation laws under the form of Partial Differential Equations (PDEs). 
Their basic way of working is the following: instead of taking $\consmomentsnumber \in \naturals^{\star}$ PDEs and directly discretize them, a \lbm scheme enlarges the size of the problem from $\consmomentsnumber$ to $\velocitynumber > \consmomentsnumber$ and treats it in a kinetic-like fashion.
This means that the  new $\velocitynumber$ variables undergo, at each time step, a local collision phase where different particle distribution functions interact, followed by a lattice-constrained stream phase where no interaction is possible.
The advantage of such idiosyncratic approach compared to more traditional numerical methods (\emph{e.g.} Finite Difference, Finite Volume, Finite Elements, \emph{etc.}) is that 
the local nature of the collision phase allows for massive parallelization of the method and the lattice-constrained stream can be computationally implemented as a pointer shift.
Although this way of proceeding is highly beneficial from a computational perspective, it yields a deficient structure to construct a clear and rigorous notion of consistency with respect to the $\consmomentsnumber$ target equations, as well as a rigorous theory of stability.
Indeed, only formal procedures, either based on the Chapman-Enskog expansion \cite{chapman1990mathematical} or on the equivalent equations by Dubois \cite{dubois2008equivalent, dubois2019nonlinear} are currently available to study the consistency of \lbm schemes.
As far as stability is concerned, most of the studies rely on the linear stability analysis of the eigenvalues of the system, see \cite{benzi1992lattice, sterling1996stability}.

In order to bridge the gap between the \lbm methods and the traditional approaches known to numerical analysts, the aim of the present contribution is to show that any \lbm scheme can be rewritten as a corresponding multi-step \fd scheme on the conserved variables, regardless of the linearity of the equilibria. 
This is made possible by developing an appropriate formalism based on commutative algebra and therefore yields a proper notion of consistency with respect to the target equations, which is that of \fd schemes (see any standard textbook such as \cite{strikwerda2004finite}).
Furthermore, we confirm that the customary \emph{von Neumann} analysis used for \lbm schemes is equivalent to performing the same analysis on the corresponding \fd scheme and is consequently particularly relevant.
The price to pay for passing from an explicit scheme with $\velocitynumber$ variables and utilizing information only at the previous time-step to a method with $\consmomentsnumber < \velocitynumber$ variables is to increase the number of previous time-steps the new solution depends on, yielding a multi-step \fd scheme. 

In the past, few authors have noticed that for some particular \lbm schemes, one has a corresponding (sometimes called ``equivalent'') \fd formulation on the conserved variables.
Despite this, no general theory has been formulated.
For instance: Suga \cite{suga2010accurate} derives  by direct computations a three-stages \fd scheme from a uni-dimensional three-velocities \scheme{1}{3} scheme,\footnote{It is customary to call \scheme{\spatialdimensionality}{\velocitynumber} a scheme in a $\spatialdimensionality$-dimensional space using $\velocitynumber$ discrete velocities.} limiting the computations to a linear framework with one relaxation parameter (SRT).
Dellacherie \cite{dellacherie2014construction} derives a two-stages \fd scheme for the \scheme{1}{2} \lbm scheme. Again, this is limited to one spatial dimension and to a linear framework.
A higher level of generality has been reached by the works of Ginzburg and collaborators, see \cite{ginzburg2009variation} for a recap. They succeeded, using a link formalism, in writing a class of Lattice Boltzmann schemes as \fd schemes \cite{d2009viscosity}.
With their highly constrained link structure to be enforced, the resulting \fd scheme with three stages is valid regardless of the spatial dimension and the choice of discrete velocities. The limitations are that the choice of moments is heavily constrained and only the case of one conserved moment is handled. Moreover, the evolution equation of the moving particles can depend on the distribution of the still particles only \emph{via} the conserved moment the equilibria depend upon and the schemes must be two-relaxation time (TRT) models with ``magic parameter'' equal to one-fourth for any link.
The difficulty in establishing a general result comes from the coupling between spatial operators and time shifts.
We must mention that during the drafting of the present contribution, an interesting work by Fu\u{c}ik and Straka \cite{fuvcik2021equivalent} has been published covering the very same subject and essentially coming to the same conclusion as our paper.
Their focus is different than ours since they adopt a purely algorithmic approach rather than a precise algebraic characterization of \lbm schemes.
We actually provide more insight into the bound on the number of time steps of the corresponding \fd scheme and our formalism, based on polynomials, aims at providing a direct link with the classical tools for the stability analysis and allows to establish a link with the Taylor expansions from \cite{dubois2019nonlinear}, as introduced in \cite{bellotti2021equivalentequations}.
In \cite{fuvcik2021equivalent}, the authors rely on a decomposition of the scheme using an hollow matrix\footnote{Matrix with zero entries on the diagonal.} yielding an equivalent form of the scheme with the diagonal non-equilibrium part, after a finite number of steps of their algorithm. However, to the best of our understanding, the origin of such algorithm is not fully clear. 
In their work, the spatial shifts of data introduced by the stream phase are taken into account using a rather cumbersome system of indices, whereas we rely on an straightforward algebraic characterization of the stream phase.

Our paper is structured as follows: in \Cref{sec:ExampleODEs}, we introduce -- in guise of friendly introduction -- the link of our problem with Ordinary Differential Equations (ODEs).
The right formalism to make \lbm schemes looking very close to a system of ODEs is provided in \Cref{sec:AlgebraicFormLBM} and allows to prove the main results of the work showcased in \Cref{sec:MainResults}.
We devote \Cref{sec:ExamplesAndSimplifications} to discuss examples, possible simplifications of the problem and particular cases deserving particular attention.
In \Cref{sec:Stability}, we prove the equivalence of the \emph{von Neumann} analysis for \lbm and \fd schemes.
In \Cref{sec:ConvergenceLBM}, we show how the well-known tools for \fd schemes can be used to prove convergence theorems for \lbm schemes. We corroborate our claim \emph{via} numerical simulations.
We eventually conclude in \Cref{sec:Conclusions}.

\section{The example of Ordinary Differential Equations}\label{sec:ExampleODEs}
  Since our way of reducing any \lbm scheme to a multi-step \fd scheme has been originally inspired by an analogy with systems of ODEs, let us introduce this way of reasoning with the following example.
  Consider the system of ODEs of size $\velocitynumber \in \naturals^{\star}$ with matrix $ \schememoments \in \matrixspace_{\velocitynumber}(\reals)$ given by 
  \begin{equation}\label{eq:SystemOfODEs}
    \begin{cases}
      \vectorial{\odevariable}'(t) &= \schememoments \vectorial{\odevariable}(t), \qquad t \geq 0, \\
      \vectorial{\odevariable}(0) &= \vectorial{\odeinitialdatum} \in \reals^{\velocitynumber}.
    \end{cases}
  \end{equation}
  Transforming a single equation of higher order into a system of first order equations like \Cref{eq:SystemOfODEs} by considering the companion matrix is a current practice, which unsurprisingly makes the problem more handy from the computational standpoint.
  Though being the analogous of what we aim at doing of \lbm schemes, the other way around, passing from a system of first order to a single equation of higher order, seems to be seldom considered.
  We proceed like in \cite{cull2005matrixdifference}. By iterating, we have that $\vectorial{\odevariable}^{(\indicetimeshift)} = \schememoments^{\indicetimeshift}\vectorial{\odevariable}$ for $\indicetimeshift \in \integerinterval{0}{ \velocitynumber}$.\footnote{We shall consistently use the notation $\integerinterval{a}{b} \definitionequality \{a, a + 1, \dots, b\}$ for $a, b \in \relatives$ and $a < b$.}
    Let $(\coeffcharact_{\indicepolynomials})_{\indicepolynomials = 0}^{\indicepolynomials = \velocitynumber} \subset \reals$ be $\velocitynumber + 1$ real coefficients, then write $\sum_{\indicepolynomials = 0}^{\indicepolynomials = \velocitynumber} \coeffcharact_k \vectorial{\odevariable}^{(k)} =  ( \sum_{k = 0}^{k = \velocitynumber} \coeffcharact_{\indicepolynomials} \schememoments^{\indicepolynomials} )\vectorial{\odevariable}$.
  Taking $(\coeffcharact_{\indicepolynomials})_{\indicepolynomials = 0}^{\indicepolynomials = \velocitynumber}$ as the coefficients of the characteristic polynomial\footnote{In the whole work, the indeterminate of any polynomial shall be denoted by $\polynomialunknown$.} $\charactpolynomial_{\schememoments} = \sum_{\indicepolynomials = 0}^{\indicepolynomials = \velocitynumber} \coeffcharact_{\indicepolynomials} \polynomialunknown^{\indicepolynomials}$ of $\schememoments$, by virtue of the Cayley-Hamilton theorem, we deduce the corresponding equation on the first variable $\odevariable_1$ (playing the role of the conserved moment), given by
  \begin{equation}
    \begin{cases}
      \sum_{\indicetimeshift = 0}^{\indicetimeshift = \velocitynumber} \coeffcharact_{\indicetimeshift} \odevariable_1^{(\indicetimeshift)}(t) &= 0, \qquad t \geq 0, \\
      \odevariable_1(0) &= (\schememoments \vectorial{\odeinitialdatum})_1, \\
      &\vdots \\
      \odevariable_1^{(\velocitynumber - 1)}(0) &= (\schememoments^{\velocitynumber - 1} \vectorial{\odeinitialdatum})_1. \\
    \end{cases} \label{eq:EquivalentODEHigherOrder}
  \end{equation}
  This provides a systematic way of performing the transformation without having to rely on hand computations and substitutions.
  To give an example, consider 
  \begin{align*}
    \schememoments_{\text{I}} =
        \begin{pmatrix}
            1 & 1 & 1 \\
            1 & 2 & 1 \\
            1 & 2 & 0
        \end{pmatrix}, \qquad \text{with} \qquad \charactpolynomial_{\schememoments_{\text{I}}} = \polynomialunknown^3 - 3\polynomialunknown^2 - 2\polynomialunknown + 1.
  \end{align*}
  Hence, the corresponding ODE on the first variable is given by $\odevariable_1''' - 3\odevariable_1'' - 2\odevariable_1' + \odevariable_1 = 0$.

  \section{Algebraic form of \lbm schemes}\label{sec:AlgebraicFormLBM}

  Now that the reader is familiar -- through a simple example -- with the main idea and the final aim of the present contribution, we introduce the general framework of \lbm schemes and the right formalism to treat them almost as systems of ODEs.

  \subsection{Spatial and temporal discretization}
    We set the problem in any spatial dimension $\spatialdimensionality = 1, 2, 3$ considering the whole space $\reals^{\spatialdimensionality}$, because we are not interested in studying boundary conditions.
    The space is discretized by a $\spatialdimensionality$-dimensional lattice $\lattice \definitionequality \spacestep \relatives^{\spatialdimensionality}$ of constant step $\spacestep > 0$ in all direction.
    The time is uniformly discretized with step $\timestep > 0$. The discrete instants of time shall be indexed by the integer indices $\timevariable \in \naturals$ so that the corresponding time is $t^{\timevariable} = \timevariable \timestep$.
    We finally introduce the so-called lattice velocity $\latticevelocity > 0$ defined by $\latticevelocity \definitionequality \spacestep/\timestep$.
    Observe that the developing theory is totally discrete and thus fully independent from the scaling between $\spacestep$ and $\timestep$.

  \subsection{Discrete velocities and shift operators}

    The first choice to be made when devising a lattice Boltzmann scheme concerns the discrete velocities $(\vectorial{\velocityletter}_{\populationindex})_{\populationindex = 1}^{\populationindex = \velocitynumber} \subset \reals^{\spatialdimensionality}$ with $\velocitynumber \in \naturals^{\star}$, which are multiples of the lattice velocity, namely $\vectorial{\velocityletter}_{\populationindex} = \latticevelocity \vectorial{\normalizedvelocityletter}_{\populationindex}$ for any $\populationindex \in \integerinterval{1}{\velocitynumber}$ with $(\vectorial{\normalizedvelocityletter}_{\populationindex})_{\populationindex = 1}^{\populationindex = \velocitynumber} \subset \relatives^{\spatialdimensionality}$. Therefore, particles are stuck to move -- at each time step -- on the lattice $\lattice$.
    We denote the distribution density of the particles moving with velocity $\vectorial{\velocityletter}_{\populationindex}$ by $\distributionletter_{\populationindex}$ for every $\populationindex \in \integerinterval{1}{\velocitynumber}$.
    The shift operators associated with the discrete velocities are an important element of the following analysis.
    \begin{definition}[Shift operator]\label{def:ShiftOperators}
      Let $\vectorial{z} \in \relatives^{\spatialdimensionality}$, then the associated shift operator on the lattice $\lattice$, denoted $\shiftoperator{\vectorial{z}}$, is defined in the following way.
      Take $\genericfunction: \lattice \to \reals$ be any function defined on the lattice,\footnote{The function could take values in any ring, see \cite{milies2002introduction}.} then the action of $\shiftoperator{\vectorial{z}}$ is

      \begin{equation*}
        (\shiftoperator{\vectorial{z}}\genericfunction)(\vectorial{\spacevariable}) = \genericfunction(\vectorial{\spacevariable} - \vectorial{z} \spacestep), \qquad \forall \vectorial{\spacevariable} \in \lattice.
      \end{equation*}
      We also introduce $\groupunitsshiftoperators \definitionequality \{\shiftoperator{\vectorial{z}} ~ \text {with} ~ \vectorial{z} \in \relatives^{\spatialdimensionality} \} \cong \relatives^{\spatialdimensionality}$.
    \end{definition}
    
    The shift yields information sought in the upwind direction with respect to the considered velocity.
    Let us introduce the natural binary operation between shifts.
    \begin{definition}[Product]\label{def:OperationsShifts}
      Let the ``product'' $\prodoperators : \groupunitsshiftoperators \times \groupunitsshiftoperators \to \groupunitsshiftoperators$ be the binary operation defined as $\shiftoperator{\vectorial{z}} \prodoperators \shiftoperator{\vectorial{w}} = \shiftoperator{\vectorial{z} + \vectorial{w}}$, for any $\vectorial{z}, \vectorial{w} \in \relatives^{\spatialdimensionality}$.
    \end{definition}
    Henceforth, the product $\prodoperators$ is understood whenever no ambiguity is possible.
    This operation provides an algebraic structure to the shifts, directly inherited from that of $\relatives^{\spatialdimensionality}$.
    \begin{proposition}
      $(\groupunitsshiftoperators, \prodoperators)$ forms an Abelian group.
    \end{proposition}
    Moreover, there is only ``one movement'' for each Cartesian direction which ``generates'' the shifts. More precisely
    \begin{align}
      \text{for} ~ \spatialdimensionality &= 1, ~ \text{let} ~ \basicx \definitionequality \shiftoperator{1}, \quad \text{then} \quad \groupunitsshiftoperators = \generatorset{\{ \basicx \}}, \label{eq:DefBasicX}\\
      \text{for} ~ \spatialdimensionality &= 2, ~ \text{let} ~ \basicx \definitionequality \shiftoperator{(1, 0)}, \basicy \definitionequality \shiftoperator{(0, 1)}, \quad \text{then} \quad \groupunitsshiftoperators = \generatorset{\{\basicx, \basicy \}}, \nonumber \\
      \text{for} ~ \spatialdimensionality &= 3, ~ \text{let} ~ \basicx \definitionequality \shiftoperator{(1, 0, 0)}, \basicy \definitionequality \shiftoperator{(0, 1, 0)}, \basicz \definitionequality \shiftoperator{(0, 0, 1)}, ~ \text{then} \quad \groupunitsshiftoperators = \generatorset{ \{\basicx, \basicy, \basicz \} }, \nonumber
    \end{align}
    where $\generatorset{\cdot}$ is the customary notation for the generating set of a group.
    We can add one more binary operation, which is non-internal to $\groupunitsshiftoperators$.
    This yields the cornerstone of this work, namely the set of \fd operators, finite combinations of weighted shifts operators \emph{via} a sum. It is defined as follows, see Chapter 3 of \cite{milies2002introduction}.
    \begin{definition}[\fd operators]\label{def:FiniteDifferenceOperators}
      The set of \fd operators on the lattice $\lattice$ is defined as

      \begin{equation*}
        \setfinitedifferenceoperators \definitionequality \reals \groupunitsshiftoperators =  \left \{ \sum\nolimits_{\genericunit \in \groupunitsshiftoperators} \coeffgenericfinitedifference_{\genericunit} \genericunit, \quad\text{where} \quad \coeffgenericfinitedifference_{\genericunit} \in \reals ~ \text{and} ~ \coeffgenericfinitedifference_{\genericunit} = 0 ~ \text{a.e.} \right  \},
      \end{equation*}
      the group ring (or group algebra) of $\groupunitsshiftoperators$ over $\reals$.
      The sum $\sumoperators : \setfinitedifferenceoperators \times \setfinitedifferenceoperators \to \setfinitedifferenceoperators$  the product\footnote{Which interestingly corresponds to the discrete convolution product.} $\prodoperators : \setfinitedifferenceoperators \times \setfinitedifferenceoperators \to \setfinitedifferenceoperators$ of two elements are defined by
        \begin{alignat*}{2}
          \left (\sum_{\genericunit \in \groupunitsshiftoperators} \alpha_{\genericunit} \genericunit \right ) &\sumoperators \left (\sum_{\genericunit \in \groupunitsshiftoperators} \beta_{\genericunit} \genericunit \right ) &&= \sum_{\genericunit \in \groupunitsshiftoperators} (\alpha_{\genericunit} + \beta_{\genericunit}) \genericunit, \\
          \left (\sum_{\genericunit \in \groupunitsshiftoperators} \alpha_{\genericunit} \genericunit \right ) &\prodoperators \left (\sum_{\mathsf{H} \in \groupunitsshiftoperators} \beta_{\mathsf{H}} \mathsf{H} \right ) &&= \sum_{\genericunit, \mathsf{H} \in \groupunitsshiftoperators} \alpha_{\genericunit}\beta_{\mathsf{H}} \genericunit \prodoperators \mathsf{H}.
        \end{alignat*}
        Furthermore, the product of $\sigma \in \reals$ with elements of $\setfinitedifferenceoperators$ is given by
        \begin{equation*}
          \sigma \left (\sum_{\genericunit \in \groupunitsshiftoperators} \alpha_{\genericunit} \genericunit \right ) = \sum_{\genericunit \in \groupunitsshiftoperators} (\sigma \alpha_{\genericunit})\genericunit .
        \end{equation*}
    \end{definition}
    With the two binary operations, $\setfinitedifferenceoperators$ behaves closely to $\relatives$, $\reals$ or $\complex$ as stated by the following result, see \cite{milies2002introduction}.
    \begin{proposition}[Ring of \fd operators]\label{prop:CommutativeRing}
      $(\setfinitedifferenceoperators, \sumoperators, \prodoperators)$ is a commutative ring.\footnote{It also an (Hopf) algebra over $\reals$ and can also be viewed as a free module where the scalars belong to $\reals$ and the basis are the elements of the group $\groupunitsshiftoperators$.}
    \end{proposition}
    Observe that $(\setfinitedifferenceoperators, \sumoperators, \prodoperators)$ is not a field: not every element of $\setfinitedifferenceoperators$ has multiplicative inverse, take for example the centered approximation of the derivative along $x$: $(\shiftoperator{-1} - \shiftoperator{1})/(2\spacestep)$ and see for instance the concept of indefinite sum in the calculus of Finite Differences \cite{milne1933calculus, miller1960introduction}.
    The elements having inverse are called ``units'' and divide all the other elements. 
    It can be easily seen that the units are the product of a non-zero real number and a shift in $\groupunitsshiftoperators$. Indeed $(\alpha \shiftoperator{\vectorial{z}})^{-1} = (1/\alpha)\shiftoperator{-\vectorial{z}}$ for any $\alpha \in \reals \smallsetminus \{ 0 \}$ and $\vectorial{z}\in \relatives^{\spatialdimensionality}$. The inverse of a unit shall also be denoted by a bar.

    \begin{remark}\label{rem:LaurentPolynomials}
      One can see $\setfinitedifferenceoperators$ as the ring of Laurent polynomials of $\spatialdimensionality$ variables over the field $\reals$, where the indeterminates are $\basicx$, $\basicy$ and $\basicz$.
      For example, for $\spatialdimensionality = 1$, the identification $\setfinitedifferenceoperators = \reals[\basicx, \basicx^{-1}] = \reals[\basicx, \conj{\basicx}]$ holds.
      This automatically implies that $\setfinitedifferenceoperators$ is more than a commutative ring, namely a unique factorization domain.
    \end{remark}
    \begin{remark}
      The reals $\reals$ can be identified with the subring $\reals \cong \{ \coeffgenericfinitedifference \shiftoperator{\vectorial{0}} ~ : ~ \coeffgenericfinitedifference \in \reals \}$.
    \end{remark}

  \subsection{Lattice Boltzmann algorithm: collide and stream}
    Any lattice Boltzmann scheme consists in an algorithm made up of two phases: a local collision phase performed on each site of the lattice and a stream phase, where particles are exchanged between different sites of the lattice. 
    Let us introduce each of them.

    \subsubsection{Collision phase}
      We adopt the point of view of the multiple-relaxation-times (MRT) schemes, where it is customary to consider the collision written as a diagonal relaxation in the moments basis, see \cite{dhumieres1992}.
      For this reason, we introduce a change of basis called moment matrix $\momentsmatrix \in \lineargroup_{\velocitynumber}(\reals)$.
      The entries of $\momentsmatrix$ can depend on $\spacestep$ and/or on $\timestep$ but cannot be a function of the space and time variables.
      Gathering the distributions into $\vectorial{\distributionletter} = \transpose{(\distributionletter_{1}, \dots, \distributionletter_{\velocitynumber})}$, the moments are recovered by $\vectorial{\momentletter} = \momentsmatrix \vectorial{\distributionletter}$.
      We also introduce
      \begin{itemize}
        \item the matrix $\matricial{\identity} \in \lineargroup_{\velocitynumber}(\reals)$ which is the identity matrix of size $\velocitynumber$;
        \item the matrix $\relaxationmatrix \in \matrixspace_{\velocitynumber}(\reals)$ is the relaxation matrix which is a singular with $\rank(\relaxationmatrix) = \velocitynumber - \consmomentsnumber$, where $\consmomentsnumber \in \integerinterval{1}{ \velocitynumber - 1}$ is the number of conserved moments:
        \begin{equation*}
          \relaxationmatrix = \diagmatrix(0, \dots, 0, \relaxparletter_{\consmomentsnumber + 1}, \dots, \relaxparletter_{\velocitynumber}),
        \end{equation*}
        where the first $\consmomentsnumber$ entries are zero\footnote{This is not always the case in literature but shall be used consistently in this paper. We put them at the beginning for the sake of presentation.} and correspond to the conserved moments, the following $\velocitynumber - \consmomentsnumber$ are such that $\relaxparletter_{\indicemoments} \in ]0, 2]$ for $\indicemoments \in \integerinterval{\consmomentsnumber + 1}{\velocitynumber}$, see \cite{dubois2008equivalent}.
        \item We employ the notation $\vectorial{\momentletter}^{\atequilibrium} \evaluationEquilibria{\timevariable} (\vectorial{\spacevariable}) = \vectorial{\momentletter}^{\atequilibrium} (\momentletter_{1}^{\timevariable}(\vectorial{\spacevariable}), \dots, \momentletter_{\consmomentsnumber}^{\timevariable}(\vectorial{\spacevariable}))$ for $\vectorial{\spacevariable} \in \lattice$, where $\vectorial{\momentletter}^{\atequilibrium}: \reals^{\consmomentsnumber} \to \reals^{\velocitynumber}$ are possibly non-linear functions of the conserved moments.
        Since these equilibria are then multiplied by $\relaxationmatrix$, the first $\consmomentsnumber$ components do not need to be defined.
      \end{itemize}
      The collision phase reads, denoting by $\collided$ any post-collision state
      \begin{equation}\label{eq:CollisionPhase}
        \vectorial{\momentletter}^{\timevariable, \collided}(\vectorial{\spacevariable}) = (\matricial{I} - \relaxationmatrix) \vectorial{\momentletter}^{\timevariable}(\vectorial{\spacevariable}) + \relaxationmatrix \vectorial{\momentletter}^{\atequilibrium} \evaluationEquilibria{\timevariable}  (\vectorial{\spacevariable}), \qquad \forall \vectorial{\spacevariable} \in \lattice.
      \end{equation}
      In the collision phase \Cref{eq:CollisionPhase}, the entries of $\relaxationmatrix$ can depend on $\spacestep$ or $\timestep$, but not on space and time. The equilibria are allowed to follow the same dependencies plus those on space and time and can also depend on some ``external variable'' like in the case of vectorial schemes \cite{graille2014approximation}.

    \subsubsection{Stream phase}
      The stream phase is diagonal in the space of the distributions. It can be written as
      \begin{equation}\label{eq:StreamPhase}
        \vectorial{\distributionletter}^{\timevariable + 1}(\vectorial{\spacevariable}) = \left ( \diagmatrix(\shiftoperator{\vectorial{\normalizedvelocityletter}_{1}}, \dots, \shiftoperator{\vectorial{\normalizedvelocityletter}_{\velocitynumber}}) \vectorial{\distributionletter}^{\timevariable, \collided}\right )(\vectorial{\spacevariable}), \qquad \forall \vectorial{\spacevariable} \in \lattice, 
      \end{equation}
      where for the first time, the matrices have entries in a commutative ring, see \cite{dummit2004abstract} and \cite{brewer1986linear}, instead than in the field $\reals$. 
      The set $\matrixspace_{\velocitynumber}(\setfinitedifferenceoperators)$ of square matrices of size $\velocitynumber$ with entries belonging to $\setfinitedifferenceoperators$ forms a ring under the usual operations between matrices. Even if $\setfinitedifferenceoperators$ is commutative  from \Cref{prop:CommutativeRing},  $\matrixspace_{\velocitynumber}(\setfinitedifferenceoperators)$ is not commutative for $\velocitynumber \geq 2$, as for real matrices and matrices of first-order differential operators \cite{dubois2019nonlinear}.

    \subsubsection{Monolithic scheme}
      The stream phase \Cref{eq:StreamPhase} can be rewritten in a non-diagonal form in the space of moments as done by \cite{dubois2019nonlinear, farag2021consistency} by introducing the matrix $\streammoments \definitionequality \momentsmatrix \diagmatrix(\shiftoperator{\vectorial{\normalizedvelocityletter}_{1}}, \dots, \shiftoperator{\vectorial{\normalizedvelocityletter}_{\velocitynumber}}) \momentsmatrix^{-1} \in \matrixspace_{\velocitynumber}(\setfinitedifferenceoperators)$ and merged with the collision phase \Cref{eq:CollisionPhase} to obtain the scheme
      \begin{equation}\label{eq:SchemeAB}
        \vectorial{\momentletter}^{\timevariable + 1}(\vectorial{\spacevariable}) = \schememoments \vectorial{\momentletter}^{\timevariable}(\vectorial{\spacevariable}) + \schemeequil \vectorial{\momentletter}^{\atequilibrium} \evaluationEquilibria{\timevariable}  (\vectorial{\spacevariable}), \qquad \forall \vectorial{\spacevariable} \in \lattice,
      \end{equation}
      where $\schememoments \definitionequality \streammoments (\vectorial{\identity} - \relaxationmatrix) \in \matrixspace_{\velocitynumber}(\setfinitedifferenceoperators)$ and $\schemeequil \definitionequality \streammoments \relaxationmatrix \in \matrixspace_{\velocitynumber}(\setfinitedifferenceoperators)$.
      In the sequel, we shall not indicate the spatial variable $\vectorial{\spacevariable} \in \lattice$ for the sake of readability.

      We observe that the operators $(\shiftoperator{\vectorial{\normalizedvelocityletter}_{\populationindex}})_{\populationindex = 1}^{\populationindex = \velocitynumber} \subset \groupunitsshiftoperators \subset  \setfinitedifferenceoperators$ are the eigenvalues of the matrix $\streammoments$. However, they are not the eigenvalues of the matrix $\schememoments$.
      Indeed, it is general false that the eigenvalues of $\schememoments$ belong to the space $\setfinitedifferenceoperators$.
      It is interesting to interpret the lattice Boltzmann scheme under the form \Cref{eq:SchemeAB} as discrete-time linear control system with matrices on a commutative ring \cite{brewer1986linear}. The moments are the state of the system evolving \emph{via} the matrix $\schememoments$, whereas the equilibria are the control \emph{via} $\schemeequil$ being a feedback observing only a part of the state, namely the conserved moments.

      We introduce our example of choice, which shall be used through the whole paper.
      \begin{example}[\scheme{1}{3} scheme with one conserved moment]\label{ex:D1Q3OneConservedVariable}
        Consider the \scheme{1}{3} scheme with one conserved moment \cite{dubois2020notion} by taking $\spatialdimensionality = 1$, $\velocitynumber = 3$ and $\consmomentsnumber = 1$. We have $\normalizedvelocityletter_1 = 0$, $\normalizedvelocityletter_2 = 1$ and $\normalizedvelocityletter_3 = -1$ with $\relaxationmatrix = \diagmatrix(0, s, p)$ and 
        \begin{displaymath}
          \momentsmatrix = 
          \begin{pmatrix}
            1 & 1 & 1 \\
            0 & \latticevelocity & -\latticevelocity \\
            -2\latticevelocity^2 & \latticevelocity^2 & \latticevelocity^2
          \end{pmatrix}, \qquad
          \streammoments = 
          \begin{pmatrix}
            \frac{1}{3}(\basicx + 1 + \conj{\basicx})& \frac{1}{2\latticevelocity}(\basicx - \conj{\basicx}) & \frac{1}{6\latticevelocity^2}(\basicx - 2 + \conj{\basicx}) \\
            \frac{\latticevelocity}{3} (\basicx - \conj{\basicx}) & \frac{1}{2}(\basicx + \conj{\basicx}) & \frac{1}{6\latticevelocity}(\basicx - \conj{\basicx}) \\
            \frac{\latticevelocity^2}{3} (\basicx - 2 + \conj{\basicx}) & \frac{\latticevelocity}{2} (\basicx - \conj{\basicx}) & \frac{1}{6}(\basicx + 2 + \conj{\basicx})
          \end{pmatrix},
        \end{displaymath}
        taking $s, p \in ]0, 2]$ and where $\basicx$ has been introduced in \Cref{eq:DefBasicX}.
        It can be used to simulate the non-linear conservation law $\partial_{t} \momentletter_1 + \partial_x \momentletter_2^{\atequilibrium} = 0$ under the acoustic scaling $\timestep \oforder \spacestep$.
        The matrices $\schememoments$ and $\schemeequil$ are
        \begin{align*}
          \schememoments &= 
          \begin{pmatrix}
            \frac{1}{3}(\basicx + 1 + \conj{\basicx})& \frac{(1-s)}{2\latticevelocity}(\basicx - \conj{\basicx}) & \frac{(1-p)}{6\latticevelocity^2}(\basicx - 2 + \conj{\basicx}) \\
            \frac{\latticevelocity}{3} (\basicx - \conj{\basicx}) & \frac{(1-s)}{2}(\basicx + \conj{\basicx}) & \frac{(1-p)}{6\latticevelocity}(\basicx - \conj{\basicx}) \\
            \frac{\latticevelocity^2}{3} (\basicx - 2 + \conj{\basicx}) & \frac{\latticevelocity(1-s)}{2} (\basicx - \conj{\basicx}) & \frac{(1-p)}{6}(\basicx + 2 + \conj{\basicx})
          \end{pmatrix},\\ 
          \schemeequil &= 
          \begin{pmatrix}
            0 & \frac{s}{2\latticevelocity}(\basicx - \conj{\basicx}) & \frac{p}{6\latticevelocity^2}(\basicx - 2 + \conj{\basicx}) \\
            0 & \frac{s}{2}(\basicx + \conj{\basicx}) & \frac{p}{6\latticevelocity}(\basicx - \conj{\basicx}) \\
            0 & \frac{\latticevelocity s}{2} (\basicx - \conj{\basicx}) & \frac{p}{6}(\basicx + 2 + \conj{\basicx})
          \end{pmatrix}.
        \end{align*}
      \end{example}

      \section{Main result of the paper}\label{sec:MainResults}

      With a new way of writing any \lbm scheme using \Cref{def:FiniteDifferenceOperators} and thanks to \Cref{prop:CommutativeRing}, which provides the ideal setting to generalize the Cayley-Hamilton theorem, we can proceed like in \Cref{sec:ExampleODEs} to prove the main result of the paper: any \lbm can be viewed as a multi-step \fd scheme on the conserved variables.
    
      \subsection{Characteristic polynomial and Cayley-Hamilton theorem}
        Polynomials with coefficients in $\setfinitedifferenceoperators$ and matrices with entries in $\setfinitedifferenceoperators$ play a central role in what we are going to develop.
        \begin{definition}[Characteristic polynomial]\label{def:CharacteristicPolynomial}
          Let $\genericcommutativering$ be a commutative ring and $\genericmatrix \in \matrixspace_{r}(\genericcommutativering)$ for some $r \in \naturals^{\star}$. 
          The characteristic polynomial of $\genericmatrix$, denoted $\charactpolynomial_{\genericmatrix} \in \genericcommutativering[\polynomialunknown]$, is given by $\charactpolynomial_{\genericmatrix}\definitionequality (-1)^r \determinant(\genericmatrix - \polynomialunknown\matricial{\identity})$, where $\determinant (\cdot)$ is the determinant and $\matricial{\identity}$ is the $r \times r$ identity matrix.
        \end{definition}
        The naive computation of the characteristic polynomial $\charactpolynomial_{\genericmatrix}$ using its definition \emph{via} the determinant could be computationally expensive, especially when dealing with symbolic computations like in our case.
        For this reason, we employed the Faddeev-Leverrier algorithm \cite{hou1998classroom} which is of polynomial complexity, generally lower than that of the pivot method.
        \begin{algorithm}
            \caption{Faddeev-Leverrier algorithm for the computation of the characteristic polynomial of a square matrix on a commutative ring $\genericcommutativering$.}\label{alg:Faddeev-Leverrier}
            \begin{algorithmic}
            \State \textbf{Input}: $\genericmatrix \in \matrixspace_{r}(\genericcommutativering)$
            \State Set $\genericmatrixtwo = \genericmatrix$
            \For{$\indicepolynomials \in \integerinterval{1}{r}$}
                \If{$\indicepolynomials > 1$}
                    \State Compute $\genericmatrixtwo = \genericmatrix (\genericmatrixtwo + \coeffcharact_{r - \indicepolynomials + 1}\matricial{I})$
                \EndIf
                \State Compute $\coeffcharact_{r-\indicepolynomials} = -\frac{\text{tr}(\genericmatrixtwo)}{\indicepolynomials}$
            \EndFor
            \State \textbf{Output}: the coefficients $(\coeffcharact_{\indicepolynomials})_{\indicepolynomials = 0}^{\indicepolynomials = r} \subset \genericcommutativering$ of the characteristic polynomial $\charactpolynomial_{\genericmatrix} = \sum_{\indicepolynomials = 0}^{\indicepolynomials = r} \coeffcharact_{\indicepolynomials} \polynomialunknown^{\indicepolynomials}$
            \end{algorithmic}
            \end{algorithm}
            The process is detailed in \Cref{alg:Faddeev-Leverrier} and only uses matrix-matrix multiplications and the computation of the trace, denoted by $\text{tr}(\cdot)$.
            \begin{example}
              Coming back to \Cref{ex:D1Q3OneConservedVariable}, it is easy to show either by manual computations or by using \Cref{alg:Faddeev-Leverrier} that $\charactpolynomial_{\schememoments} = \polynomialunknown^3 + \coeffcharact_2 \polynomialunknown^2 + \coeffcharact_1 \polynomialunknown + \coeffcharact_0$ with
              \begin{align*}
                \coeffcharact_2 &= p(\basicx + 4 + \conj{\basicx})/6 + s(\basicx + \conj{\basicx})/2 - (\basicx + 1 + \conj{\basicx}), \\
                                &= -(1-p)(\basicx + 4 + \conj{\basicx})/6 - (1-s)(\basicx + \conj{\basicx})/2 - (\basicx + 1 + \conj{\basicx})/3,\\
                \coeffcharact_1 &= ps(\basicx + 1 + \conj{\basicx})/3 - p(5\basicx + 2 + 5\conj{\basicx})/6 - s(\basicx + 2 + \conj{\basicx})/2 + (\basicx + 1 + \conj{\basicx}), \\
                                &= \left ( -p(1-s)/3 - (p + s - 2)/2\right )(\basicx + \conj{\basicx}) + 2 \left ( (1-s) - s(1-p) - (p+s - 2) \right ),\\
              \coeffcharact_0 &= -(1 - p)(1 - s). 
              \end{align*}
              We see that $\coeffcharact_0 = 0$ if either $s$ or $p$ are equal to one, this shall be discussed in \Cref{sec:RelaxationOnEquilibria}.
              On the other hand $\coeffcharact_1 = 0$ if we have $s = p = 1$.
            \end{example}
          
            A central result used in this work is the Cayley-Hamilton theorem for matrices over a commutative ring, see \cite{brewer1986linear} for the proof, generalizing the same result holding for matrices on a field utilized in \Cref{sec:ExampleODEs}.
            \begin{theorem}[Cayley-Hamilton]\label{thm:CayleyHamilton}
              Let $\genericcommutativering$ be a commutative ring and $\genericmatrix \hspace{-0.05cm} \in \hspace{-0.05cm} \matrixspace_{r}(\genericcommutativering)$ for some $r \in \naturals^{\star}$. Then $\charactpolynomial_{\genericmatrix}$ is a monic polynomial in the ring $\genericcommutativering[\polynomialunknown]$ in the indeterminate $\polynomialunknown$, under the form $\charactpolynomial_{\genericmatrix} = \polynomialunknown^r + \coeffcharact_{r-1}\polynomialunknown^{r-1} + \dots \coeffcharact_1 \polynomialunknown + \coeffcharact_0$
              with $(\coeffcharact_{\indicepolynomials})_{\indicepolynomials=0}^{\indicepolynomials=r} \subset \genericcommutativering$. Then\footnote{Sometimes, we shall indulge to the notation $\charactpolynomial_{\genericmatrix}(\genericmatrix) = \matricial{0}$.} $\genericmatrix^r + \coeffcharact_{r-1}\genericmatrix^{r-1} + \dots + \coeffcharact_1 \genericmatrix + \coeffcharact_0 \matricial{\identity}= \matricial{0}$.
            \end{theorem}
            This result states that any square matrix with entries in a commutative ring verifies its characteristic equation.
          
          \subsection{Corresponding \fd schemes}
        
          \begin{figure}[h]
            \begin{center}
              \def\svgwidth{0.7\textwidth}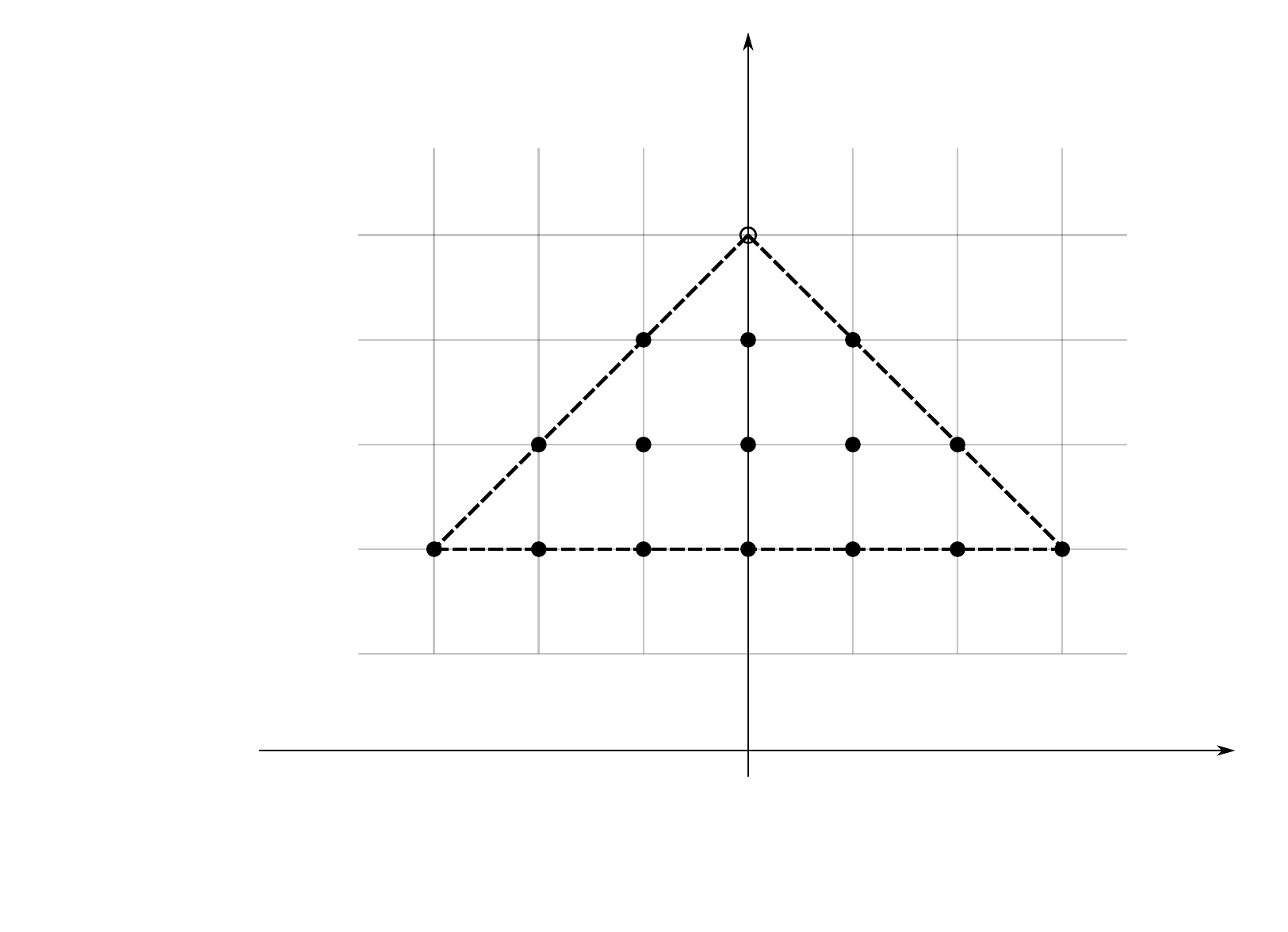
            \end{center}\caption{\label{fig:coneofinfluence}Maximal space-time domain of dependence of the corresponding \fd scheme for $\consmomentsnumber = 1$ (full black points inside the grey area) by virtue of \Cref{prop:ReductionFiniteDifferenceGeneral} in the case of $\spatialdimensionality = 1$. The maximal space-time slopes are determined by the maximal shift of the considered scheme whereas the number of involved time-steps is at most $\velocitynumber + 1$.}
          \end{figure}
            
            The previous \Cref{thm:CayleyHamilton} is the key for proving the following results, whose backbone is essentially the same than in \Cref{sec:ExampleODEs}.
        
            \subsubsection{One conserved moment}
            We first analyze the case of one conserved moment, namely $\consmomentsnumber = 1$, to keep the presentation as simple as possible.
            We shall eventually deal with $\consmomentsnumber > 1$ once the principles are established. 
        
            \begin{proposition}[Corresponding \fd scheme for $\consmomentsnumber = 1$]\label{prop:ReductionFiniteDifferenceGeneral}
              Let $\consmomentsnumber = 1$, then the \lbm scheme \Cref{eq:SchemeAB} corresponds to a multi-step explicit \fd scheme on the conserved moment $\momentletter_1$ under the form
              \begin{displaymath}
                \momentletter_{1}^{\timevariable + 1} = -\sum_{\indicetimeshift = 0}^{\velocitynumber - 1} \coeffcharact_{\indicetimeshift} \momentletter_{1}^{\timevariable + 1 - \velocitynumber + \indicetimeshift} + \left (\sum_{\indicetimeshift = 0}^{\velocitynumber-1} \left ( \sum_{\ell=0}^{\indicetimeshift} \coeffcharact_{\velocitynumber + \ell - \indicetimeshift} \schememoments^{\ell} \right ) \schemeequil \vectorial{\momentletter}^{\atequilibrium} \evaluationEquilibria{\timevariable - \indicetimeshift}  \right )_1,
              \end{displaymath}
              where $(\coeffcharact_{\indicepolynomials})_{\indicepolynomials = 0}^{\indicepolynomials = \velocitynumber}\subset \setfinitedifferenceoperators$ are the coefficients of $\charactpolynomial_{\schememoments} = \sum_{\indicepolynomials = 0}^{\indicepolynomials = \velocitynumber} \coeffcharact_{\indicepolynomials} \polynomialunknown^{\indicepolynomials}$, the characteristic polynomial of $\schememoments$.
            \end{proposition}
            This result means that the conserved moment satisfies an explicit multi-step \fd scheme with at most $\velocitynumber$ steps, thus involving $\velocitynumber + 1$ discrete time instants, see \Cref{fig:coneofinfluence}.
            The maximal size of spatial influence at each past time step can be deduced by looking at \Cref{alg:Faddeev-Leverrier}, derived from the Newton's identities.
        
            It is interesting to observe that also the non-conserved moments satisfy a \fd numerical scheme, see the following proof. However, these schemes would depend on the conserved moment \emph{via} the equilibria and are therefore not independent from the rest of the system.
            \begin{proof}
              Let $\timevariable \in \naturals$. Then for any $\indicetimeshift \in \naturals$, applying \Cref{eq:SchemeAB} recursively we have
              \begin{equation*}
                \vectorial{\momentletter}^{\timevariable + 1} = \schememoments^{\indicetimeshift} \vectorial{\momentletter}^{\timevariable -(\indicetimeshift-1)} + \sum_{\ell = 0}^{\indicetimeshift-1} \schememoments^{\ell} \schemeequil  \vectorial{\momentletter}^{\atequilibrium}\evaluationEquilibria{\timevariable - \ell}.
              \end{equation*}
              We perform a temporal shift in order to fix the first term on the right hand side regardless of the value of $\indicetimeshift$.
              Introduce $\tilde{\timevariable} \definitionequality \timevariable -(\indicetimeshift-1)$, therefore
              \begin{equation*}
                \vectorial{\momentletter}^{\tilde{\timevariable} + \indicetimeshift} = \schememoments^{\indicetimeshift} \vectorial{\momentletter}^{\tilde{\timevariable}} + \sum_{\ell = 0}^{\indicetimeshift-1} \schememoments^{\ell} \schemeequil  \vectorial{\momentletter}^{\atequilibrium}\evaluationEquilibria{\tilde{\timevariable} + \indicetimeshift - 1 - \ell}.
              \end{equation*}
              This holds true, in particular, for any $\indicetimeshift \in  \integerinterval{0}{\velocitynumber}$.
              We can then consider the coefficients $(\coeffcharact_{\indicepolynomials})_{\indicepolynomials = 0}^{\indicepolynomials = \velocitynumber}$ of the characteristic polynomial $\charactpolynomial_{\schememoments} = \sum_{\indicepolynomials = 0}^{\indicepolynomials = \velocitynumber} \coeffcharact_{\indicepolynomials} \polynomialunknown^{\indicepolynomials}$ of $\schememoments$ and write
              \begin{equation*}
                \sum_{\indicetimeshift = 0}^{\velocitynumber} \coeffcharact_{\indicetimeshift} \vectorial{\momentletter}^{\tilde{\timevariable} + \indicetimeshift} = \left ( \sum_{\indicetimeshift = 0}^{\velocitynumber} \coeffcharact_{\indicetimeshift}  \schememoments^{\indicetimeshift} \right ) \vectorial{\momentletter}^{\tilde{\timevariable}} + \sum_{\indicetimeshift = 0}^{\velocitynumber} \coeffcharact_{\indicetimeshift}  \left ( \sum_{\ell = 0}^{\indicetimeshift-1} \schememoments^{\ell} \schemeequil  \vectorial{\momentletter}^{\atequilibrium}\evaluationEquilibria{\tilde{\timevariable} + \indicetimeshift - 1 - \ell} \right ).
              \end{equation*}
              Applying the Cayley-Hamilton \Cref{thm:CayleyHamilton} by virtue of \Cref{prop:CommutativeRing}, we know that $ \sum_{\indicepolynomials = 0}^{\indicepolynomials = \velocitynumber} \coeffcharact_{\indicepolynomials}  \schememoments^{\indicepolynomials} = \matricial{0}$.
              Using the monicity of the characteristic polynomial and coming back by setting $\tilde{\timevariable} + \velocitynumber = \timevariable + 1$ gives
              \begin{equation*}
                \vectorial{\momentletter}^{\timevariable + 1} = -\sum_{\indicetimeshift = 0}^{\velocitynumber - 1} \coeffcharact_{\indicetimeshift} \vectorial{\momentletter}^{\timevariable + 1 - \velocitynumber  + \indicetimeshift} + \sum_{\indicetimeshift = 0}^{\velocitynumber} \coeffcharact_{\indicetimeshift}  \left ( \sum_{\ell = 0}^{\indicetimeshift-1} \schememoments^{\ell} \schemeequil  \vectorial{\momentletter}^{\atequilibrium}\evaluationEquilibria{\timevariable - \velocitynumber + \indicetimeshift - \ell} \right ).
              \end{equation*}
              The last sum can start from $\indicetimeshift = 1$. Performing a change of indices in the last double sum yields the result.
              \begin{equation}\label{eq:ReducedFDOverAllVariables}
                \vectorial{\momentletter}^{\timevariable + 1} = -\sum_{\indicetimeshift = 0}^{\velocitynumber - 1} \coeffcharact_{\indicetimeshift} \vectorial{\momentletter}^{\timevariable + 1 - \velocitynumber  + \indicetimeshift} + \sum_{\indicetimeshift = 0}^{\velocitynumber-1} \left ( \sum_{\ell=0}^{\indicetimeshift} \coeffcharact_{\velocitynumber + \ell - \indicetimeshift} \schememoments^{\ell} \right ) \schemeequil \vectorial{\momentletter}^{\atequilibrium}\evaluationEquilibria{\timevariable - \indicetimeshift}.
              \end{equation}
            \end{proof}

            \begin{example}
              We come back to \Cref{ex:D1Q3OneConservedVariable}. Using \Cref{prop:ReductionFiniteDifferenceGeneral}, we have the corresponding \fd scheme given by
              \begin{align}
                \momentletter_1^{\timevariable + 1} = &-\frac{p}{6}(\basicx + 4 + \conj{\basicx}) \momentletter_1^{\timevariable} - \frac{s}{2}(\basicx + \conj{\basicx}) \momentletter_1^{\timevariable}  + (\basicx + 1 + \conj{\basicx}) \momentletter_1^{\timevariable} -\frac{ps}{3}(\basicx + 1 + \conj{\basicx}) \momentletter_1^{\timevariable-1} \nonumber \\
                &+ \frac{p}{6}(5\basicx + 2 + 5\conj{\basicx}) \momentletter_1^{\timevariable-1} + \frac{s}{2}(\basicx + 2 + \conj{\basicx}) \momentletter_1^{\timevariable-1} - (\basicx + 1 + \conj{\basicx}) \momentletter_1^{\timevariable-1} \nonumber \\
                &+(1 - p)(1 - s)\momentletter_1^{\timevariable-2} + \frac{s}{2\latticevelocity} (\basicx - \conj{\basicx}) \momentletter_2^{\atequilibrium}\evaluationEquilibria{\timevariable} - \frac{s (1-p)}{2\latticevelocity} (\basicx - \conj{\basicx}) \momentletter_2^{\atequilibrium}\evaluationEquilibria{\timevariable - 1} \nonumber \\
                &+\frac{p}{6\latticevelocity^2} (\basicx - 2 + \conj{\basicx}) \momentletter_3^{\atequilibrium}\evaluationEquilibria{\timevariable} + \frac{p (1-s)}{6\latticevelocity^2} (\basicx - 2 + \conj{\basicx}) \momentletter_3^{\atequilibrium}\evaluationEquilibria{\timevariable - 1}. \label{eq:DFSchemeD1Q3Example}
              \end{align}
              One can easily check its consistency -- under the acoustic scaling -- with the target conservation law.
            \end{example}
        
            \begin{remark}
              One could think of allowing $\momentsmatrix$ and/or $\relaxationmatrix$ to depend on the space and time variables. This would imply to consider weights made up of functions instead of the real numbers in \Cref{def:FiniteDifferenceOperators}.
              However, $\setfinitedifferenceoperators$ would no longer be commutative, because the multiplication by a function does not commute with shifts (not shift-invariant according to  \cite{rota1973foundations}). For example, take $\vectorial{z} \in \relatives^{\spatialdimensionality}$ and a function $g : \lattice \to \reals$, then
              \begin{align*}
                \left ( \left (\shiftoperator{\vectorial{z}} \prodoperators \left ( g \shiftoperator{\vectorial{0}} \right ) \right )\genericfunction \right )(\vectorial{\spacevariable}) &= g(\vectorial{\spacevariable} - \vectorial{z} \spacestep) \genericfunction (\vectorial{\spacevariable} - \vectorial{z} \spacestep), \\
                \left ( \left (\left ( g \shiftoperator{\vectorial{0}}  \right ) \prodoperators \shiftoperator{\vectorial{z}} \right )\genericfunction \right )(\vectorial{\spacevariable}) &= g(\vectorial{\spacevariable}) \genericfunction (\vectorial{\spacevariable} - \vectorial{z} \spacestep),
              \end{align*}
              for every $\vectorial{\spacevariable} \in \lattice$ and for any function $\genericfunction : \lattice \to \reals$.
              The right-hand sides are not equal in general, except if $g$ is constant.
            \end{remark}
        
            \subsubsection{Several conserved moments and vectorial schemes}
        
            Consider now to deal with multiple conservation laws, namely $\consmomentsnumber > 1$.
            We select a conserved moment and we consider the other conserved moments as ``slave'' variables as the equilibria have been until so far, for $\consmomentsnumber = 1$, because they imply variables that we eventually want to keep.
            In particular, we utilize different polynomials for different conserved moments to obtain the \fd schemes.
            To formalize this concept, for any square matrix $\genericmatrix \in \matrixspace_{\velocitynumber}(\setfinitedifferenceoperators)$, consider $\cutmatrixsquare{\genericmatrix}{\setoflines} \definitionequality  ( \sum_{\indiceslines \in \setoflines} \canonicalbasisvector_{\indiceslines} \otimes \canonicalbasisvector_{\indiceslines}  ) \genericmatrix   ( \sum_{\indiceslines \in \setoflines} \canonicalbasisvector_{\indiceslines} \otimes \canonicalbasisvector_{\indiceslines}  ) \in \matrixspace_{\velocitynumber}(\setfinitedifferenceoperators)$ for any $\setoflines \subset \integerinterval{1}{\velocitynumber}$, corresponding to the matrix where only the rows and columns of indices $\setoflines$ are conserved and the remaining ones are set to zero. 
            We can also consider the matrix $\cutmatrixsquaretrimmed{\genericmatrix}{\setoflines} \in \matrixspace_{|\setoflines|\times |\setoflines|}(\setfinitedifferenceoperators)$ obtained by keeping only the rows and the columns indexed in $\setoflines$.
            A useful corollary of \Cref{thm:CayleyHamilton} and of the Laplace formula for the determinant is the following.
            \begin{corollary}\label{Cor:CutMatricesPolynomials}
              Let $\genericmatrix \in \matrixspace_{\velocitynumber}(\setfinitedifferenceoperators)$ and $\setoflines \subset \integerinterval{1}{\velocitynumber}$, then one has that $\charactpolynomial_{\cutmatrixsquare{\genericmatrix}{\setoflines}} = \polynomialunknown^{\velocitynumber - \setcardinality{\setoflines}}  \charactpolynomial_{\cutmatrixsquaretrimmed{\genericmatrix}{\setoflines}}$.
              Moreover, the polynomial $\charactpolynomial_{\cutmatrixsquaretrimmed{\genericmatrix}{\setoflines}}$ annihilates $\cutmatrixsquare{\genericmatrix}{\setoflines}$.
            \end{corollary}
            This means that the characteristic polynomial of $\cutmatrixsquare{\genericmatrix}{\setoflines}$ is directly linked to that of the smaller matrix $\cutmatrixsquaretrimmed{\genericmatrix}{\setoflines}$, which is thus faster to compute, and that the latter is an annihilator for the first matrix.

            For any conserved moment indexed by $\indiceconserved \in \integerinterval{1}{\consmomentsnumber}$ we introduce the matrix $\schememoments_{\indiceconserved} \definitionequality \cutmatrixsquare{\schememoments}{\{ \indiceconserved\} \cup \integerinterval{\consmomentsnumber + 1}{\velocitynumber}}$ and $\schememomentsother_{\indiceconserved} \definitionequality \cutmatrixsquare{\schememoments}{\integerinterval{1}{\consmomentsnumber} \smallsetminus \{ \indiceconserved\}}$.
            Notice that we have the decomposition $\schememoments = \schememoments_{\indiceconserved} + \schememomentsother_{\indiceconserved} $.
            Indeed, we ``save'' the conserved moments other than the \thup{$\indiceconserved$} by placing them into $ \schememomentsother_{\indiceconserved}$, which shall not participate in the computation of the characteristic polynomial.
            With this notations, we have generated a family of problems from \Cref{eq:SchemeAB} under the form
            \begin{equation}\label{eq:SchemeABL}
              \vectorial{\momentletter}^{\timevariable + 1} = \schememoments_{\indiceconserved} \vectorial{\momentletter}^{\timevariable} + \schememomentsother_{\indiceconserved}\vectorial{\momentletter}^{\timevariable} + \schemeequil \vectorial{\momentletter}^{\atequilibrium}\evaluationEquilibria{\timevariable}, \qquad \indiceconserved \in \integerinterval{1}{\consmomentsnumber}.
            \end{equation}
            It is useful to stress that the term $\schememoments_{\indiceconserved} \vectorial{\momentletter}^{\timevariable}$ in \Cref{eq:SchemeABL} does not involve any conserved moment other than the \thup{$\indiceconserved$}.
            Conversely, $\schememomentsother_{\indiceconserved}\vectorial{\momentletter}^{\timevariable}$ does not involve any function except the conserved moments other than the \thup{$\indiceconserved$}.
            Then, the corresponding \fd schemes come under the form stated by the following Proposition.
            \begin{proposition}[Corresponding \fd scheme for $\consmomentsnumber \geq 1$]\label{prop:ReductionVectorialorNl1scheme}
              Let $\consmomentsnumber \geq 1$, then the \lbm scheme \Cref{eq:SchemeAB} rewritten as \Cref{eq:SchemeABL} corresponds to the multi-step explicit \fd schemes on the conserved moments $\momentletter_1, \dots, \momentletter_{\consmomentsnumber}$ under the form
              \begin{align*}
                \momentletter_{\indiceconserved}^{\timevariable + 1} = -\sum_{\indicetimeshift = 0}^{\velocitynumber - \consmomentsnumber} \coeffcharact_{\indiceconserved, \indicetimeshift} \momentletter_{\indiceconserved}^{\timevariable - \velocitynumber + \consmomentsnumber + \indicetimeshift} &+ \left ( \sum_{\indicetimeshift = 0}^{\velocitynumber-\consmomentsnumber} \left ( \sum_{\ell=0}^{\indicetimeshift} \coeffcharact_{\indiceconserved, \velocitynumber +1 - \consmomentsnumber + \ell - \indicetimeshift} \schememoments_{\indiceconserved}^{\ell} \right ) \schememomentsother_{\indiceconserved}\vectorial{\momentletter}^{\timevariable - \indicetimeshift} \right )_{\indiceconserved} \\
                &+ \left ( \sum_{\indicetimeshift = 0}^{\velocitynumber-\consmomentsnumber} \left ( \sum_{\ell=0}^{\indicetimeshift} \coeffcharact_{\indiceconserved, \velocitynumber + 1 - \consmomentsnumber + \ell - \indicetimeshift} \schememoments_{\indiceconserved}^{\ell} \right ) \schemeequil \vectorial{\momentletter}^{\atequilibrium}\evaluationEquilibria{\timevariable - \indicetimeshift} \right )_{\indiceconserved},
              \end{align*}
              for any $\indiceconserved \in \integerinterval{1}{\consmomentsnumber}$ where $(\coeffcharact_{\indiceconserved, \indicepolynomials})_{\indicepolynomials = 0}^{\indicepolynomials = \velocitynumber + 1 -\consmomentsnumber}\subset \setfinitedifferenceoperators$ are the coefficients of the characteristic polynomial $\charactpolynomial_{\schememoments_{\indiceconserved}} = \polynomialunknown^{\consmomentsnumber - 1} \sum_{\indicepolynomials = 0}^{\indicepolynomials = \velocitynumber + 1 - \consmomentsnumber}\coeffcharact_{\indiceconserved, \indicepolynomials} \polynomialunknown^{\indicepolynomials} $ of $\schememoments_{\indiceconserved}$.
            \end{proposition}
            This Proposition states that for each conserved moment, the corresponding \fd scheme has at most $\velocitynumber - \consmomentsnumber$ steps, thus involves $\velocitynumber - \consmomentsnumber + 1$ discrete times.
            This result encompasses and generalizes \Cref{prop:ReductionFiniteDifferenceGeneral}.
            The proof is the same than that of \Cref{prop:ReductionFiniteDifferenceGeneral} by taking advantage of \Cref{Cor:CutMatricesPolynomials}.
            We show in another contribution \cite{bellotti2021equivalentequations} that the result of \Cref{prop:ReductionVectorialorNl1scheme} is the right one to bridge between the consistency analysis of \fd schemes and the Taylor expansions on the \lbm schemes for $\consmomentsnumber \geq 1$ proposed by \cite{dubois2019nonlinear}.
          
            \begin{example}[\scheme{1}{3} for two conservation laws]\label{ex:TwoConservedVariables}
              Consider the \scheme{1}{3} scheme \cite{bellotti2021multiresolution} with $\spatialdimensionality = 1$, $\velocitynumber = 3$ and $\normalizedvelocityletter_1 = 0$, $\normalizedvelocityletter_2 = 1$ and $\normalizedvelocityletter_3 = -1$
              \begin{equation}\label{eq:D1Q3TwoConservedMoment}
                \momentsmatrix = \left ( \begin{matrix}
                                     1 & 1 & 1 \\
                                     0 & \latticevelocity & - \latticevelocity \\
                                     0 & \latticevelocity^2 & \latticevelocity^2
                                    \end{matrix} \right ), \qquad
                \relaxationmatrix = \diagmatrix(0, 0, p),
                \quad \text{with} \quad p \neq 1,
              \end{equation}
              thus having $\consmomentsnumber = 2$.
              This scheme can be used to simulate the system of conservation laws $ \partial_t \momentletter_1 + \partial_x \momentletter_2 = 0$ and $\partial_t \momentletter_2 + \partial_x \momentletter_3^{\atequilibrium} = 0$ under the acoustic scaling $\timestep \oforder \spacestep$.
              Using \Cref{prop:ReductionVectorialorNl1scheme} we have 
              \begin{align*}
                \momentletter_1^{\timevariable + 1} &= \momentletter_1^{\timevariable} + \frac{1}{2}(1-p)(\basicx + \conj{\basicx})\momentletter_1^{\timevariable} - \frac{1}{2}(1-p)(\basicx + \conj{\basicx})\momentletter_1^{\timevariable-1} + \frac{(\basicx - \conj{\basicx})}{2\latticevelocity} \momentletter_2^{\timevariable} \\
                &-\frac{(1-p)(\basicx -\conj{\basicx})}{2\latticevelocity} \momentletter_2^{\timevariable-1} + \frac{p(\basicx - 2 + \conj{\basicx})}{2\latticevelocity^2} \momentletter_3^{\atequilibrium}\evaluationEquilibria{\timevariable}, \\
                \momentletter_2^{\timevariable + 1} &=\frac{1}{2}(2-p)(\basicx + \conj{\basicx})\momentletter_2^{\timevariable}- (1-p)\momentletter_2^{\timevariable-1} + \frac{p(\basicx - \conj{\basicx})}{2\latticevelocity} \momentletter_3^{\atequilibrium}\evaluationEquilibria{\timevariable}.
              \end{align*}
              One could remark that the linear part is different from one scheme to the other, since we have used different polynomials for each conserved moment.
            \end{example}

            \subsection{Initialization schemes}
              In the corresponding \fd schemes in \Cref{prop:ReductionFiniteDifferenceGeneral} and \Cref{prop:ReductionVectorialorNl1scheme}, the only remaining freedom is to devise the initialization schemes for the multi-step schemes at regime, analogously to \Cref{eq:EquivalentODEHigherOrder} in \Cref{sec:ExampleODEs}.
              This is the counter-part of the freedom of choice on the initial data for the original \lbm scheme, which are not necessarily taken at equilibrium, see \cite{junk2015l2}.
              By applying \Cref{eq:SchemeAB} to the initial data as many times as needed, one progressively obtains the initialization schemes, as function of the initial datum.
              It is worthwhile observing that the choice of initial datum does not play any role in the previous procedure and does not influence the stability analysis of \Cref{sec:Stability}. It only comes into play during the consistency analysis of the numerical method, which is not investigated in this paper, in particular, as far as time boundary layers are concerned, see \cite{van2009smooth, rheinlander2007analysis}.

              \section{Examples, simplifications and particular cases}\label{sec:ExamplesAndSimplifications}
  
              Now that the main results of the paper, namely \Cref{prop:ReductionFiniteDifferenceGeneral} and \Cref{prop:ReductionVectorialorNl1scheme}, have been stated and proved, we can analyze and comment some particular cases which deserve a closer look.
              More examples are available in the Appendices.
            
              \begin{example}[ODEs]\label{ex:EDOs}
                To illustrate some basic peculiarities that easily transpose to \lbm schemes, we introduce the following matrices extending the discussion of \Cref{sec:ExampleODEs}.
              \begin{align*}
                \schememoments_{\text{II}} =
                    \begin{pmatrix}
                        1 & 1 & 1 \\
                        0 & 2 & 0 \\
                        0 & 0 & 2
                    \end{pmatrix}, \qquad
                    \schememoments_{\text{III}} =
                    \begin{pmatrix}
                        1 & 1 & 0 \\
                        1 & 2 & 0 \\
                        1 & 2 & 1
                    \end{pmatrix}, \qquad
                    \schememoments_{\text{IV}} =
                    \begin{pmatrix}
                        1 & 0 & 1 \\
                        0 & -2 & 0 \\
                        0 & 0 & 2
                    \end{pmatrix}.
              \end{align*}
            
              For $\schememoments_{\text{II}}$, we have $\charactpolynomial_{\schememoments_{\text{II}}} = \polynomialunknown^3 - 5\polynomialunknown^2 + 8\polynomialunknown - 4$, corresponding to $\odevariable_1''' - 5\odevariable_1'' + 8\odevariable_1' - 4\odevariable_1 = 0$.
              However, contrarily to $\schememoments_{\text{I}}$ in \Cref{sec:ExampleODEs}, the characteristic polynomial $\charactpolynomial_{\schememoments_{\text{II}}}$ does not correspond to the minimal polynomial $\minimalpolynomial_{\schememoments_{\text{II}}} = \polynomialunknown^2 - 3\polynomialunknown + 2$. Thus in this case, we could use the latter to obtain \Cref{eq:EquivalentODEHigherOrder} having $\odevariable_1'' - 3\odevariable_1' + 2 \odevariable_1 = 0$.
              This phenomenon is studied in \Cref{sec:MinimalReductionTimeSteps}.
              It indicates that we can achieve a more compact corresponding ODE by using the annihilating polynomial of smallest degree on every variable.
              This does not change the core of the strategy.

              For $\schememoments_{\text{III}}$, we obtain $\charactpolynomial_{\schememoments_{\text{III}}} = \polynomialunknown^3 - 4\polynomialunknown^2 + 4\polynomialunknown - 1$, corresponding to $\odevariable_1''' - 4\odevariable_1'' + 4\odevariable_1' - \odevariable_1 = 0$. However, by inspecting $\schememoments_{\text{III}}$, one notices that the first two equations do not depend on the last variable $\odevariable_3$.
              For this reason, we could have considered the matrix $\cutmatrixsquaretrimmed{\schememoments_{\text{III}}}{ \{1,2 \} }$ obtained from $\schememoments_{\text{III}}$ by removing the last row and column. In this case $\charactpolynomial_{\cutmatrixsquaretrimmed{\schememoments_{\text{III}}}{ \{1,2 \} }}= \polynomialunknown^2 - 3\polynomialunknown + 1$, corresponding to the equation $\odevariable_1'' - 3 \odevariable_1' + \odevariable_1 = 0$. 
              This kind of situation for \lbm schemes is investigated in \Cref{sec:RelaxationOnEquilibria}.
              It is interesting to observe once more that $\charactpolynomial_{\cutmatrixsquaretrimmed{\schememoments_{\text{III}}}{ \{1,2 \} }}$ divides $ \charactpolynomial_{{\schememoments}_{\text{III}}}$.
              This shows that an initial inspection of the matrix can yield a reduction of the size of the problem that can be achieved by a simple trimming operation, which eliminates some variable from the problem but treats the remaining ones as usual.
            
              Finally, consider $\schememoments_{\text{IV}}$. In this case the characteristic polynomial and the minimal polynomial coincide $\charactpolynomial_{\schememoments_{\text{IV}}}= \polynomialunknown^3 - \polynomialunknown^2 - 4\polynomialunknown + 4$ corresponding to the equation $\odevariable_1''' - \odevariable_1'' - 4\odevariable_1' + 4\odevariable_1 = 0$.
              However, if we take the polynomial $\minimalannpolynomial_{\schememoments_{\text{IV}}}= \polynomialunknown^2  -3 \polynomialunknown + 2$ such that $\minimalannpolynomial_{\schememoments_{\text{IV}}}$ divides $\charactpolynomial_{\schememoments_{\text{IV}}} = \minimalpolynomial_{\schememoments_{\text{IV}}}$ and such that 
              \begin{equation*}
                \minimalannpolynomial_{\schememoments_{\text{IV}}} (\schememoments_{\text{IV}}) = 
                \begin{pmatrix}
                  0 & 0 & 0 \\
                  0 & 12 & 0 \\
                  0 & 0 & 0
                \end{pmatrix},
              \end{equation*}
              we see that it annihilates the first row, thus can be used instead of the other polynomials to yield \Cref{eq:EquivalentODEHigherOrder}. This gives $\odevariable_1'' -3\odevariable_1' + 2\odevariable_1 = 0$. The question is elucidated for \lbm schemes in \Cref{sec:DifferentReductiontStrategy} and show that asking for the annihilation of the whole matrix is too much to achieve a restatement of the equation focusing only on the first variable.
              This strategy is different from the previous one because not all the lines of the matrix are treated in the same way.
              \end{example}
            
              Let us transpose these observations to actual \lbm schemes.
              A question which might arise concerns the possibility of performing better than the characteristic polynomial, in terms of number of steps in the resulting \fd scheme.
              There are cases, which seem quite rare according to our experience (we succeeded in finding only one special case where this happens), where the answer is positive.
              This phenomenon has also been discussed by \cite{fuvcik2021equivalent}, without envisioning a systematic way of guaranteeing the minimality of the \fd scheme obtained by their algorithm.
            
              \subsection{Minimal reductions in terms of time-steps}\label{sec:MinimalReductionTimeSteps}
            
              The first idea to obtain a simpler scheme is to use the minimal polynomial of $\schememoments$ (or its submatrices, if needed) as done for $\schememoments_{\text{II}}$ in \Cref{ex:EDOs}.
              \begin{definition}[Minimal polynomial]
                Let $\genericcommutativering$ be a commutative ring and $\genericmatrix \in \matrixspace_{r}(\genericcommutativering)$ for some $r \in \naturals^{\star}$. 
                We define the minimal polynomial of $\genericmatrix$, denoted $\minimalpolynomial_{\genericmatrix}$ as being the monic polynomial in $\genericcommutativering[\polynomialunknown]$ of smallest degree, thus under the form
                \begin{equation*}
                  \minimalpolynomial_{\genericmatrix} = \polynomialunknown^{\degree( \minimalpolynomial_{\genericmatrix})} + \coeffminimal_{\degree( \minimalpolynomial_{\genericmatrix})-1}\polynomialunknown^{\degree( \minimalpolynomial_{\genericmatrix})-1} + \dots + \coeffminimal_1 \polynomialunknown + \coeffminimal_0,
                \end{equation*}
                with $(\coeffminimal_{\indicepolynomials})_{\indicepolynomials=0}^{\indicepolynomials=\degree(\minimalpolynomial_{\genericmatrix})} \subset \genericcommutativering$ such that
                \begin{equation*}
                  \genericmatrix^{\degree(\minimalpolynomial_{\genericmatrix})} + \coeffminimal_{\degree(\minimalpolynomial_{\genericmatrix})-1}\genericmatrix^{\degree(\minimalpolynomial_{\genericmatrix})-1} + \dots + \coeffminimal_1 \genericmatrix + \coeffminimal_0 \matricial{\identity}= \matricial{0}.
                \end{equation*}
            \end{definition}  
            The characteristic and the minimal polynomial for problems set of a commutative ring are linked by a divisibility property.
            \begin{lemma}\label{lemma:MinimalVsCharacteristic}
              Let $\genericcommutativering$ be a commutative ring and $\genericmatrix \in \matrixspace_{r}(\genericcommutativering)$ for some $r \in \naturals^{\star}$, then $\minimalpolynomial_{\genericmatrix}$ divides $\charactpolynomial_{\genericmatrix}$.
              Therefore, we also have $\degree(\minimalpolynomial_{\genericmatrix}) \leq \degree(\charactpolynomial_{\genericmatrix})$.
            \end{lemma}
            \begin{proof}
              The proof  is standard and works the same than that of Lemma \ref{lemma:divisibility}.
            \end{proof}
            Unfortunately, the minimal polynomial cannot be mechanically computed by something like \Cref{alg:Faddeev-Leverrier} as for the characteristic polynomial.
            The same reduction of \Cref{prop:ReductionFiniteDifferenceGeneral} with $\degree(\minimalpolynomial_{\schememoments})$ instead of $\velocitynumber$ and $\coeffminimal_{\indicepolynomials}$ instead of $\coeffcharact_{\indicepolynomials}$ is possible.
            It can be observed that for \Cref{ex:D1Q3OneConservedVariable}, the minimal and the characteristic polynomial of the matrix $\schememoments$ coincide.
            We have been unable to find an example of \lbm scheme where the minimal polynomial does not match the characteristic polynomial.
            
              \subsection{Relaxation on the equilibrium}\label{sec:RelaxationOnEquilibria}
              
                Secondly, a more careful look at relaxation matrix allows us to write it as $\relaxationmatrix = \diagmatrix(0, \dots, 0, \relaxparletter_{\consmomentsnumber + 1}, \dots, \relaxparletter_{\consmomentsnumber + \nontrivialnumber}, 1, \dots, 1)$, where $\relaxparletter_{\indicemoments} \in ]0, 1[ \cup ]1, 2]$ for $\indicemoments \in \integerinterval{\consmomentsnumber + 1}{\consmomentsnumber + \nontrivialnumber}$ for some $\nontrivialnumber \in \naturals$ and the last $\velocitynumber - \nontrivialnumber - \consmomentsnumber$ relaxation parameters are equal to one, meaning that the corresponding moments exactly relax on their respective equilibrium.
                Without loss of generality, we have decided to put them at the end of $\relaxationmatrix$.
                The fact of considering some relaxation rates equal to one is used in the so-called ``regularization'' models, see \cite{coreixas2019comprehensive} and references therein, showing the enhancement of the stability features of the schemes.
            
                In terms of matrix structure, the consequence is that the last $\velocitynumber - \consmomentsnumber - \nontrivialnumber$ columns of $\schememoments$ are zero, analogously to $\schememoments_{\text{III}}$ in \Cref{ex:EDOs}.
                We can therefore employ the following decomposition of $\schememoments$: $\schememoments = \cutmatrixsquare{\schememoments}{\integerinterval{1}{\consmomentsnumber + \nontrivialnumber}} + \cutmatrixsquare{\schememoments}{\integerinterval{\consmomentsnumber + \nontrivialnumber + 1}{\velocitynumber}}$ similarly to \Cref{eq:SchemeABL}.
                We shall consider the characteristic polynomial of $\cutmatrixsquaretrimmed{\schememoments}{\integerinterval{1}{\consmomentsnumber + \nontrivialnumber}}$ (if $\consmomentsnumber = 1$, otherwise the characteristic polynomials of its submatrices), whereas we know that the second matrix does not involve the last $\velocitynumber - \consmomentsnumber - \nontrivialnumber$ moments (indeed, non conserved) because the corresponding columns are zero.
                Therefore, \Cref{prop:ReductionFiniteDifferenceGeneral} and \Cref{prop:ReductionVectorialorNl1scheme} are still valid using $\consmomentsnumber + \nontrivialnumber$ instead of $\velocitynumber$ and the matrix $\cutmatrixsquaretrimmed{\schememoments}{\integerinterval{1}{\consmomentsnumber + \nontrivialnumber}}$ instead of $\schememoments$.
                The corresponding \fd scheme for each conserved moment shall therefore have at most $\nontrivialnumber + 1$ steps instead of $\velocitynumber + 1$.

                \begin{example}\label{ex:D1Q3OneConservedVariableP1}
                  We come back to \Cref{ex:D1Q3OneConservedVariable} taking $p = 1$ and $s \neq 1$, thus having $\nontrivialnumber = 1$ and $\consmomentsnumber = 1$.
                  Following the procedure described before gives $\charactpolynomial_{\cutmatrixsquaretrimmed{\schememoments}{\integerinterval{1}{2}}} = \polynomialunknown^2 + \coeffcharact_1 \polynomialunknown + \coeffcharact_0$ with $\coeffcharact_1 = - (1-s)(\basicx + \conj{\basicx})/2 - (\basicx + 1 + \conj{\basicx})/3$ and $\coeffcharact_0 = (1-s)(\basicx + 4 + \conj{\basicx})/6$ and the corresponding scheme
                  \begin{align*}
                    \momentletter_1^{\timevariable + 1} = & \frac{(1-s)}{2}(\basicx + \conj{\basicx}) \momentletter_1^{\timevariable}  +\frac{1}{3} (\basicx + 1 + \conj{\basicx})\momentletter_1^{\timevariable} - \frac{(1-s)}{6}(\basicx + 4 + \conj{\basicx})\momentletter_1^{\timevariable-1} \\
                    &+ \frac{s}{2\latticevelocity} (\basicx - \conj{\basicx}) \momentletter_2^{\atequilibrium}\evaluationEquilibria{\timevariable} +\frac{1}{6\latticevelocity^2} (\basicx - 2 + \conj{\basicx}) \momentletter_3^{\atequilibrium}\evaluationEquilibria{\timevariable} + \frac{(1-s)}{6\latticevelocity^2} (\basicx - 2 + \conj{\basicx}) \momentletter_3^{\atequilibrium}\evaluationEquilibria{\timevariable - 1}.
                  \end{align*}
                  Unsurprisingly, this is \Cref{eq:DFSchemeD1Q3Example} setting $p = 1$, obtained treating a smaller problem.
                \end{example}
                Observe that the fact of taking all the relaxation rates equal to one, relaxing on the equilibria, is the core mechanism of the relaxation schemes \cite{bouchut2004nonlinear}.
                In this case, there is nothing to do since the original \lbm scheme is already in the form of a \fd scheme on the conserved moments.
                Our way of proposing a corresponding \fd scheme using characteristic polynomials is flawlessly compatible with this setting.

              \subsection{A different reduction strategy}\label{sec:DifferentReductiontStrategy}
              The third idea is to proceed as for $\schememoments_{\text{IV}}$ in \Cref{ex:EDOs}, namely looking for a polynomial which does not annihilate the whole matrix $\schememoments$.
              To simplify the presentation, we limit ourselves to $\consmomentsnumber = 1$, namely one conserved moment.
              We sketch this strategy to account for previous results on the subject \cite{d2009viscosity, ginzburg2009variation}.
              Nevertheless, we shall justify its limited interest at the end of the Section.
              
              \begin{example}[Link scheme with magic parameter]\label{ex:Link}
                Consider the so-called link scheme by \cite{d2009viscosity, ginzburg2009variation} defined for any spatial dimension $\spatialdimensionality = 1, 2, 3$ considering $\velocitynumber = 1 + 2\numbercoupleslinkscheme$ with $\numbercoupleslinkscheme \in \naturals^{\star}$ with $\vectorial{\normalizedvelocityletter}_1 = \vectorial{0}$ and any $\vectorial{\normalizedvelocityletter}_{2\populationindex} = -\vectorial{\normalizedvelocityletter}_{2\populationindex + 1} \neq \vectorial{0}$ for $\populationindex \in \integerinterval{1}{\numbercoupleslinkscheme}$.
                The system is taken with all the so-called ``magic parameters'' equal to one-fourth, therefore $\relaxationmatrix = \diagmatrix(0, s, 2-s, s, 2-s, \dots) \in \matrixspace_{1+2\numbercoupleslinkscheme}(\reals)$ for $s \neq 1$ and
                \begin{equation*}
                  \momentsmatrix = 
                  \left(\begin{array}{@{}c|cc|cc|cc@{}}
                  1 & 1 & 1 & \cdots & \cdots & 1 & 1 \\ \hline
                  0 & \latticevelocity & -\latticevelocity & 0 & 0 & 0 & 0 \\
                  0 & \latticevelocity^2 & \latticevelocity^2 & 0 & 0 & 0 & 0 \\ \hline
                  \vdots & 0 & 0 & \ddots & \ddots &  0 & 0 \\
                  \vdots & 0 & 0 & \ddots & \ddots &  0 & 0 \\ \hline
                  0 & 0 & 0 & 0 & 0 & \latticevelocity & -\latticevelocity \\
                  0 & 0 & 0 & 0 & 0 & \latticevelocity^2 & \latticevelocity^2 \\
                  \end{array}\right) \in \matrixspace_{1+2\numbercoupleslinkscheme}(\reals),
                \end{equation*}
                The claim in \cite{ginzburg2009variation} is that the corresponding \fd scheme is the two-steps scheme
                \begin{align}
                  \momentletter_1^{\timevariable+1} = (2-s) \momentletter_1^{\timevariable} - (1-s)\momentletter_1^{\timevariable-1} &+ s \left ( \sum_{\ell = 1}^{\numbercoupleslinkscheme} \frac{(\shiftoperator{\vectorial{\normalizedvelocityletter}_{2\ell}}-\shiftoperator{-\vectorial{\normalizedvelocityletter}_{2\ell}})}{2\latticevelocity} \momentletter_{2\ell}^{\atequilibrium}\evaluationEquilibria{\timevariable} \right ) \nonumber \\
                  &+ \frac{(2-s)}{2} \left ( \sum_{\ell = 1}^{\numbercoupleslinkscheme} \frac{(\shiftoperator{\vectorial{\normalizedvelocityletter}_{2\ell}} - 2 +\shiftoperator{-\vectorial{\normalizedvelocityletter}_{2\ell}})}{\latticevelocity^2} \momentletter_{2\ell + 1}^{\atequilibrium}\evaluationEquilibria{\timevariable}\right ). \label{eq:EquivalentFDLinkScheme}
                \end{align}
                This is true regardless of the choice of $\spatialdimensionality$ and $\numbercoupleslinkscheme$.
                By direct inspection of the corresponding \fd scheme \Cref{eq:EquivalentFDLinkScheme}, we can say that this reduction has been achieved using the polynomial $\minimalannpolynomial_{\schememoments}= \polynomialunknown^2 - (2 - s) \polynomialunknown + (1-s)$. However, it can be easily shown that this polynomial does not annihilate the entire matrix $\schememoments$ as the minimal and characteristic polynomials do: it only does so for the first row.
              \end{example}
              Indeed, we have seen for ODEs in \Cref{ex:EDOs} that we might try just to annihilate the first row of the problem.
              Thus, we define the polynomial annihilating all the first row of the matrix $\schememoments$, except the very first element.
              \begin{definition}\label{def:MPAMFR}
                We call $\minimalannpolynomial_{\schememoments} \in \setfinitedifferenceoperators[\polynomialunknown]$ ``minimal polynomial annihilating most of the first row'' (MPAMFR) of $\schememoments$ the monic polynomial of minimal degree under the form
                \begin{equation*}
                  \minimalannpolynomial_{\schememoments} = \polynomialunknown^{\degree(\minimalannpolynomial_{\schememoments})} + \coeffannminimal_{\degree(\minimalannpolynomial_{\schememoments})-1} \polynomialunknown^{\degree(\minimalannpolynomial_{\schememoments})-1}  + \dots + \coeffannminimal_{1} \polynomialunknown + \coeffannminimal_0,
                \end{equation*}
                with $(\coeffannminimal_{\indicepolynomials})_{\indicepolynomials = 0}^{\indicepolynomials = \degree(\minimalannpolynomial_{\schememoments})} \subset \setfinitedifferenceoperators$ such that for every $\indicescolumns \in \integerinterval{2}{\velocitynumber}$
                \begin{equation*}
                  (\schememoments^{\degree(\minimalannpolynomial_{\schememoments})})_{1\indicescolumns} + \coeffannminimal_{\degree(\minimalannpolynomial_{\schememoments})-1} (\schememoments^{\degree(\minimalannpolynomial_{\schememoments})-1})_{1\indicescolumns} + \dots + \coeffannminimal_1 (\schememoments)_{1\indicescolumns} = 0.
                \end{equation*}
              \end{definition}
              By seeing the coefficients of this unknown polynomial as the unknowns of a linear system, the problem of finding $\minimalannpolynomial_{\schememoments}$ can be rewritten in terms of matrices.\footnote{The same procedure is used to find the minimal polynomial, since we do not have a definition like \Cref{def:CharacteristicPolynomial}.}
              Let $K \in \integerinterval{1}{\degree(\minimalpolynomial_{\schememoments})}$ and construct the matrix of variable size  
              \begin{equation}\label{eq:MatrixVariableSize}
                \matrixminannpol_{K} = 
                \begin{pmatrix}
                    (\schememoments)_{12} & \cdots & (\schememoments^K)_{12} \\
                    \vdots & & \vdots \\
                    (\schememoments)_{1, \nontrivialnumber +1} & \cdots & (\schememoments^K)_{1, \nontrivialnumber +1} \\
                \end{pmatrix}\in \matrixspace_{\nontrivialnumber \times K}(\setfinitedifferenceoperators).
              \end{equation}
              Therefore, we want to find the smallest $K \in \integerinterval{1}{\degree(\minimalpolynomial_{\schememoments})}$ such that $\kernel(\matrixminannpol_{K}) \neq \{ \vectorial{0} \}$, that is, the smallest $K \in \integerinterval{1}{\degree(\minimalpolynomial_{\schememoments})}$   such that $\matrixminannpol_{K}$ is not injective.
              Since the kernel of the ``minimal'' $\matrixminannpol_{K}$ shall be a $\setfinitedifferenceoperators$-module of dimension $1$, we can chose a monic polynomial by always taking $\coeffannminimal_K = 1$.
              On the other hand, it should be observed that the zero order coefficient $\coeffannminimal_0$ remains free.
              This underdetermination comes from the fact that we do not request that $\minimalannpolynomial_{\schememoments}$ annihilates the whole first row.
              \begin{proposition}\label{prop:ReductionFiniteDifferenceAnnhilationFirstRowOnly}
                Let $\consmomentsnumber = 1$, then the lattice Boltzmann scheme \eqref{eq:SchemeAB} can be rewritten as a \fd scheme on the conserved moment $\momentletter_1$ under the form
                \begin{align}
                  {\momentletter}_1^{{\timevariable} + 1} = &- \sum_{\indicetimeshift = 1}^{\degree(\minimalannpolynomial_{\schememoments}) - 1} \coeffannminimal_{\indicetimeshift} {\momentletter}_1^{{\timevariable} +1 -  \degree(\minimalannpolynomial_{\schememoments}) + \indicetimeshift} + \left ( \sum_{\indicetimeshift = 1}^{\degree(\minimalannpolynomial_{\schememoments})} \coeffannminimal_{\indicetimeshift} (\schememoments^{\indicetimeshift})_{11} \right ) \momentletter_1^{\timevariable+1- \degree(\minimalannpolynomial_{\schememoments})}  \nonumber \\
                  &+\left ( \sum_{\indicetimeshift = 0}^{\degree(\minimalannpolynomial_{\schememoments}) - 1} \left ( \sum_{\ell = 0}^{\indicetimeshift} \coeffannminimal_{\degree(\minimalannpolynomial_{\schememoments}) + \ell - \indicetimeshift} \schememoments^{\ell} \right ) \schemeequil \vectorial{\momentletter}^{\atequilibrium}\evaluationEquilibria{\timevariable - \indicetimeshift} \right )_1, \label{eq:FiniteDifferenceReduced1}
                \end{align}
                where $(\coeffannminimal_{\indicepolynomials})_{\indicepolynomials = 1}^{\indicepolynomials = \degree(\minimalannpolynomial_{\schememoments})} \subset \setfinitedifferenceoperators$ are the coefficients of $\minimalannpolynomial_{\schememoments} = \sum_{\indicepolynomials = 0}^{\indicepolynomials = \degree(\minimalannpolynomial_{\schememoments})} \coeffannminimal_{\indicepolynomials} \polynomialunknown^{\indicepolynomials}$.
              \end{proposition} 
              The proof can be found in the Appendices.
              Looking at \Cref{eq:FiniteDifferenceReduced1}, we see that we do not need the value of $\coeffannminimal_0$ to reduce the scheme, neither to reduce $\schememoments$ nor to deal with the equilibria through $\schemeequil$.
              Changing time indices and putting everything on the left hand side
              \begin{align*}
                   \sum_{\indicetimeshift = 1}^{\degree(\minimalannpolynomial_{\schememoments})} \coeffannminimal_{\indicetimeshift} {\momentletter}_1^{\tilde{\timevariable} + \indicetimeshift} - \left ( \sum_{\indicetimeshift = 1}^{\degree(\minimalannpolynomial_{\schememoments})} \coeffannminimal_{\indicetimeshift} (\schememoments^k)_{11} \right ) \momentletter^{\tilde{\timevariable}} &= \sum_{\indicetimeshift = 0}^{\degree(\minimalannpolynomial_{\schememoments})} \tilde{\coeffannminimal}_{\indicetimeshift} {\momentletter}_1^{\tilde{\timevariable} + \indicetimeshift}, \\ 
                   &= \sum_{\indicetimeshift = 1}^{\degree(\minimalannpolynomial_{\schememoments})} \coeffannminimal_{\indicetimeshift}  \left ( \sum_{\ell = 0}^{\indicetimeshift-1} \schememoments^{\ell} \schemeequil  \vectorial{\momentletter}^{\atequilibrium}\evaluationEquilibria{\tilde{\timevariable} + \indicetimeshift - 1 - \ell} \right )_1, \\
                   &= \sum_{\indicetimeshift = 1}^{\degree(\minimalannpolynomial_{\schememoments})} \tilde{\coeffannminimal}_{\indicetimeshift}  \left ( \sum_{\ell = 0}^{\indicetimeshift-1} \schememoments^{\ell} \schemeequil  \vectorial{\momentletter}^{\atequilibrium}\evaluationEquilibria{\tilde{\timevariable} + \indicetimeshift - 1 - \ell} \right )_1,
              \end{align*}
              where we have defined
              \begin{equation*}
                  \tilde{\coeffannminimal}_{\indicepolynomials} = 
                  \begin{cases}
                      \coeffannminimal_{\indicepolynomials}, \qquad &\indicepolynomials \in \integerinterval{1}{ \degree(\minimalannpolynomial_{\schememoments})}, \\
                      -\sum_{\ell = 1}^{\ell = \degree(\minimalannpolynomial_{\schememoments})} \coeffannminimal_{\ell} (\schememoments^{\ell})_{11}, \qquad &\indicepolynomials = 0.
                  \end{cases}
              \end{equation*}
              This generates a polynomial, which is indeed $\minimalannpolynomial_{\schememoments}$ but with a precise choice of $\coeffannminimal_0$.
              We will soon give a precise characterization of this particular polynomial.
              \begin{definition}
                We call $\minimalannrowpolynomial_{\schememoments} \in \setfinitedifferenceoperators[\polynomialunknown]$ ``minimal polynomial annihilating the first row'' (MPAFR) of $\schememoments$ the monic polynomial of minimal degree under the form
                \begin{equation*}
                  \minimalannrowpolynomial_{\schememoments} = \polynomialunknown^{\degree(\minimalannrowpolynomial_{\schememoments})} + \tilde{\coeffannminimal}_{\degree(\minimalannrowpolynomial_{\schememoments})-1} \polynomialunknown^{\degree(\minimalannrowpolynomial_{\schememoments})-1}  + \dots + \tilde{\coeffannminimal}_{1} \polynomialunknown^{1} + \tilde{\coeffannminimal}_0,
                \end{equation*}
                such that for every $\indicescolumns \in \integerinterval{1}{\velocitynumber}$
                \begin{equation}\label{eq:tmpAnnhihilating}
                  (\schememoments^{\degree(\minimalannrowpolynomial_{\schememoments})})_{1\indicescolumns} + \tilde{\coeffannminimal}_{\degree(\minimalannrowpolynomial_{\schememoments})-1} (\schememoments^{\degree(\minimalannrowpolynomial_{\schememoments})-1})_{1\indicescolumns} + \dots + \tilde{\coeffannminimal}_1 (\schememoments)_{1\indicescolumns} + \tilde{\coeffannminimal}_0 = 0.
                \end{equation}
              \end{definition}
              Compared to \Cref{def:MPAMFR}, we are just asking the property to hold also for the very first element of the first row, namely for $\indicescolumns = 1$.
              This polynomial is $\minimalannpolynomial_{\schememoments}$ for a particular choice of $\coeffannminimal_0$. 
              It has been deduced from the reduction of the lattice Boltzmann scheme.
              \begin{lemma}
                The polynomial of degree $\degree(\minimalannpolynomial_{\schememoments})$ given by
                \begin{equation*}
                  \minimalannrowpolynomial_{\schememoments} = \polynomialunknown^{\degree(\minimalannpolynomial_{\schememoments})} + \coeffannminimal_{\degree(\minimalannpolynomial_{\schememoments})-1} \polynomialunknown^{\degree(\minimalannpolynomial_{\schememoments})-1} + \dots + \coeffannminimal_1 \polynomialunknown -\sum_{l = 1}^{\degree(\minimalannpolynomial_{\schememoments})} \coeffannminimal_l (\schememoments^l)_{11},
                \end{equation*}
                where $(\coeffannminimal_{\indicepolynomials})_{\indicepolynomials = 1}^{\indicepolynomials = \degree(\minimalannpolynomial_{\schememoments})} \subset \setfinitedifferenceoperators$ are the coefficients of a MPAMFR $\minimalannpolynomial_{\schememoments}$ of $\schememoments$ being $\minimalannpolynomial_{\schememoments} = \sum_{\indicepolynomials = 0}^{\indicepolynomials = \degree(\minimalannpolynomial_{\schememoments})} \coeffannminimal_{\indicepolynomials} \polynomialunknown^{\indicepolynomials}$, is the MPAFR $\minimalannrowpolynomial_{\schememoments}$ of $\schememoments$.
              \end{lemma}
              \begin{proof}
                We are only left to check \Cref{eq:tmpAnnhihilating} for $\indicescolumns = 1$.
              \end{proof}
            
              So in order to reduce the lattice Boltzmann scheme to a \fd scheme using the new strategy, considering a MPAMFR or the MPAFR is exactly the same thing.
              Moreover, the MPAFR (but not the more general MPAMFR) can be linked to the minimal/characteristic polynomial.\footnote{The principle is the same than the one linking the characteristic and the minimal polynomial through divisibility.}
              \begin{lemma}\label{lemma:divisibility}
                Let $\minimalpolynomial_{\schememoments} \in \setfinitedifferenceoperators[\polynomialunknown]$ be the minimal polynomial of $\schememoments$, then $\minimalannrowpolynomial_{\schememoments}$ exists and divides the minimal polynomial $\minimalpolynomial_{\schememoments}$. Moreover $\degree(\minimalannrowpolynomial_{\schememoments}) = \degree(\minimalannpolynomial_{\schememoments}) \leq \degree(\minimalpolynomial_{\schememoments})$.
              \end{lemma}
              The proof is given in the Appendices.
              We now show how this discussion allows to account for \Cref{ex:Link} and more specifically for \Cref{eq:EquivalentFDLinkScheme}.
              \begin{example}
                We come back to \Cref{ex:Link}. 
                We introduce the notations $\avglink_{\ell} \definitionequality \shiftoperator{\vectorial{\normalizedvelocityletter}_{2\ell}} + \shiftoperator{\vectorial{\normalizedvelocityletter}_{2\ell + 1}}$, the ``average'' on the \thup{$\ell$} link and $\difflink_{\ell} \definitionequality \shiftoperator{\vectorial{\normalizedvelocityletter}_{2\ell}} - \shiftoperator{\vectorial{\normalizedvelocityletter}_{2\ell + 1}}$, the ``difference'' on the \thup{$\ell$} link, for any $\ell \in \integerinterval{1}{\numbercoupleslinkscheme}$.
                Elementary computations show that
                \begin{equation*}
                  \matrixminannpol_{2} =
                  \begin{pmatrix}
                   \frac{(1-s)\difflink_1}{2\latticevelocity} & \frac{(1-s)(2-s)\difflink_1}{2\latticevelocity} \\
                   -\frac{(1-s)(\avglink_1-2)}{2\latticevelocity^2} & -\frac{(1-s)(2-s)(\avglink_1 - 2)}{2\latticevelocity^2} \\ 
                   \frac{(1-s)\difflink_2}{2\latticevelocity} & \frac{(1-s)(2-s)\difflink_2}{2\latticevelocity} \\
                   -\frac{(1-s)(\avglink_2-2)}{2\latticevelocity^2} & -\frac{(1-s)(2-s)(\avglink_2 - 2)}{2\latticevelocity^2} \\ 
                   \vdots & \vdots \\
                   \frac{(1-s)\difflink_{\numbercoupleslinkscheme}}{2\latticevelocity} & \frac{(1-s)(2-s)\difflink_{\numbercoupleslinkscheme}}{2\latticevelocity} \\
                   -\frac{(1-s)(\avglink_{\numbercoupleslinkscheme}-2)}{2\latticevelocity^2} & -\frac{(1-s)(2-s)(\avglink_{\numbercoupleslinkscheme} - 2)}{2\latticevelocity^2}
                  \end{pmatrix} \in \matrixspace_{(2\numbercoupleslinkscheme) \times 2} (\setfinitedifferenceoperators).
                \end{equation*}
                The equations have the same structure for every $2 \times 2$ block: thus we can find a solution by studying each block if it turns out that the solution does not depend on the block indices. Let  $\ell \in \integerinterval{1}{\numbercoupleslinkscheme}$. We want to solve for non-trivial $\coeffannminimal_{1}, \coeffannminimal_{2}$ such that
                \begin{equation*}
                  \begin{cases}
                  \frac{(1-s)\difflink_{\ell}}{2\latticevelocity}\coeffannminimal_{1} + \frac{(1-s)(2-s)\difflink_{\ell}}{2\latticevelocity}\coeffannminimal_{2} &= 0, \\
                  -\frac{(1-s)(\avglink_{\ell}-2)}{2\latticevelocity^2}\coeffannminimal_{1}  -\frac{(1-s)(2-s)(\avglink_{\ell} - 2)}{2\latticevelocity^2}\coeffannminimal_{2} &= 0,
                  \end{cases}
                \end{equation*}
                thus we clearly see that the solution is $\coeffannminimal_{1} = -(2-s)\coeffannminimal_{2}$, but we can pick $\coeffannminimal_{2} = 1$ to have a monic polynomial. Therefore $\coeffannminimal_{1} = -(2-s)$ independently from $\ell$.
                Thus, the polynomial $\minimalannpolynomial_{\schememoments} = \polynomialunknown^2 - (2-s) \polynomialunknown + \coeffannminimal_0$. Picking $\coeffannminimal_0 = - \coeffannminimal_2 (\schememoments^2)_{11} - \coeffannminimal_1 (\schememoments)_{11} = -1 + (2-s) = 1-s$ yields the polynomial $\minimalannrowpolynomial_{\schememoments}$ as previously seen.
              \end{example}
            
              This approach correctly recovers the result from \cite{ginzburg2009variation} following a different path.
              However, to our understanding, this new strategy is of moderate interest since it relies on an \emph{ad hoc} and problem-dependent procedure \Cref{eq:MatrixVariableSize} which can be practically exploited only for highly constrained systems, see \Cref{ex:Link} or for schemes of modest size.
              Moreover, for general schemes, it yields the same result than \Cref{prop:ReductionFiniteDifferenceGeneral} using the characteristic polynomial (take \Cref{ex:D1Q3OneConservedVariable} for instance) but passing from an inefficient approach to the computation of the polynomial instead of using the more performant \Cref{alg:Faddeev-Leverrier}.
            
              \subsection{Conclusion and future perspectives}
              Beyond the divisibility property \Cref{prop:ReductionFiniteDifferenceAnnhilationFirstRowOnly}, the fact of not utilizing the characteristic polynomial with its explicit \Cref{def:CharacteristicPolynomial} constitutes -- due to the previously highlighted lack of generality -- an obstruction to show the link with the Taylor expansions \cite{dubois2019nonlinear}, as we did in \cite{bellotti2021equivalentequations}.  
              We therefore stress once more the interest of the general formulations by \Cref{prop:ReductionFiniteDifferenceGeneral} and \Cref{prop:ReductionVectorialorNl1scheme}, which shall allow to enlighten the issue of the stability of the schemes, as in the following Section.

              \section{Stability}\label{sec:Stability}
              Arguably, the \emph{von Neumann} analysis is the most widely used technique to investigate the stability of \lbm schemes. Though employed for any number $\consmomentsnumber$ of conserved moments, we shall consider it only for $\consmomentsnumber = 1$, to keep mathematical rigour.
              The \emph{von Neumann} analysis consists in the linearization of the problem around an equilibrium state \cite{sterling1996stability}, followed by the rewrite of the scheme using the Fourier transform and the study of the spectrum of the derived matrix.
              Unsurprisingly, this is also common in the framework of \fd methods, see Chapter 4 in \cite{gustafsson2013time} and Chapter 4 in \cite{strikwerda2004finite}.
              We observe that the linear $\elltwospace$ stability, though being widespread, is not the only possible one for \lbm schemes: the interested reader can refer to  \cite{junk2009weighted, junk2015l2} for the $\elltwospace$-weighted stability, to \cite{caetano2019result} for the $\ellpspace{1}$ stability and finally to \cite{dubois2020notion} for the $\ellpspace{\infty}$ stability.
              Future efforts shall be dedicated to the investigation of the impact of \Cref{prop:ReductionFiniteDifferenceGeneral} and \Cref{prop:ReductionVectorialorNl1scheme} on these different notions of stability.
            
              \subsection{Fourier analysis}
              
              We briefly introduce the Fourier analysis on lattices following Chapter 2 of \cite{strikwerda2004finite}.
              We define $\fouriertransformletter: \ell^2(\lattice) \cap \ell^1(\lattice) \to \elltwospace([-\pi/\spacestep, \pi / \spacestep]^{\spatialdimensionality})$, called Fourier transform, defined as follows.
              Let $\genericfunction \in \ell^2(\lattice) \cap \ell^1(\lattice)$, then
              \begin{equation*}
                \fouriertransform{\genericfunction}(\vectorial{\freqvariable}) \definitionequality \frac{1}{(2\pi)^{\spatialdimensionality/2}} \sum_{\vectorial{\spacevariable} \in \lattice} e^{-\imag \vectorial{\spacevariable} \cdot \vectorial{\freqvariable}} \genericfunction(\vectorial{\spacevariable}), \qquad \vectorial{\freqvariable} \in \left [ -\frac{\pi}{\spacestep}, \frac{\pi}{\spacestep}\right ]^{\spatialdimensionality}.
              \end{equation*}
              In this Section, the regularity assumptions shall hold for any function.
              The Fourier transform is extended to less regular entities by density arguments.
              The interest of the Fourier transform lies in the fact that it is an isometry, thanks to the Parseval's identity \cite{strikwerda2004finite} and that it allows to represent the action of operators acting \emph{via} the convolution product (also called filters) like the \fd operators $\setfinitedifferenceoperators$ as a multiplication on $\complex$.
              We can therefore represent any shift operator in the Fourier space.
              \begin{lemma}[Shift operator in the Fourier space]\label{lemma:ShiftFourierSpace}
                Let $\vectorial{z} \in \relatives^{\spatialdimensionality}$ and $\genericfunction \in \ell^2(\lattice) \cap \ell^1(\lattice)$, then
                \begin{equation*}
                  \fouriertransform{\shiftoperator{\vectorial{z}}\genericfunction}(\vectorial{\freqvariable}) = e^{-\imag \spacestep \vectorial{z} \cdot \vectorial{\freqvariable}} \fouriertransform{\genericfunction}(\vectorial{\freqvariable}), \qquad \vectorial{\freqvariable} \in \left [ -\frac{\pi}{\spacestep}, \frac{\pi}{\spacestep}\right ]^{\spatialdimensionality}.
                \end{equation*}
                Therefore, the representation of the shift operator $\shiftoperator{\vectorial{z}}$ in the Fourier space is $\shiftoperatorfourier{\vectorial{z}} \definitionequality e^{-\imag \spacestep \vectorial{z} \cdot \vectorial{\freqvariable}}$ and acts multiplicatively.
              \end{lemma}
              \begin{proof}
                Let $\genericfunction: \lattice \to \reals$ with $\genericfunction \in \ell^2 (\lattice) \cap \ell^1(\lattice)$. We have, for every wave number $\vectorial{\freqvariable} \in [-\pi/\spacestep, \pi/\spacestep]^{\spatialdimensionality}$
                \begin{align*}
                  \fouriertransform{\shiftoperator{\vectorial{z}}\genericfunction}(\vectorial{\freqvariable}) &= \frac{1}{(2\pi)^{\spatialdimensionality/2}} \sum_{\vectorial{\spacevariable} \in \lattice} e^{-\imag \vectorial{\spacevariable} \cdot \vectorial{\freqvariable}} \genericfunction(\vectorial{\spacevariable} - \vectorial{z}\spacestep), \\
                    &= \frac{1}{(2\pi)^{\spatialdimensionality/2}} \sum_{\vectorial{y} \in \lattice} e^{-\imag (\vectorial{y} + \vectorial{z}\spacestep) \cdot \vectorial{\freqvariable}} \genericfunction(\vectorial{y}) = e^{-\imag \spacestep \vectorial{z} \cdot \vectorial{\xi}}  \fouriertransform{\genericfunction}(\vectorial{\freqvariable}).
                \end{align*}
            \end{proof}
              The rewrite of $\groupunitsshiftoperators$ and $\setfinitedifferenceoperators$ in the Fourier space is done in the straightforward manner, namely
              \begin{equation*}
                \groupunitsshiftoperatorsfourier \definitionequality \left \{ \shiftoperatorfourier{\vectorial{z}} = e^{-\imag \spacestep \vectorial{z} \cdot \vectorial{\freqvariable}} ~ \text{with} ~ \vectorial{z} \in \relatives^{\spatialdimensionality} \right \}, \qquad
                \setfinitedifferenceoperatorsfourier \definitionequality \reals \groupunitsshiftoperatorsfourier,
              \end{equation*}
              where the sum and the products are the standard ones on $\complex$.
              All that has been said for $\setfinitedifferenceoperators$ holds for the new representation in the Fourier space $\setfinitedifferenceoperatorsfourier$. 
              Indeed, for any $\genericfinitedifference = \sum_{\genericunit \in \groupunitsshiftoperators} \coeffgenericfinitedifference_{\genericunit} \genericunit \in \setfinitedifferenceoperators$, we indicate $\fourier{\genericfinitedifference} \definitionequality \sum_{{\genericunit} \in \groupunitsshiftoperators} \coeffgenericfinitedifference_{\genericunit} \fourier{\genericunit} \in \setfinitedifferenceoperatorsfourier$ its representative in the Fourier space. 
              Considering a matrix $\genericmatrix \in \matrixspace_{\velocitynumber}(\setfinitedifferenceoperators)$, its Fourier representation $\fourier{\genericmatrix} \in \matrixspace_{\velocitynumber}(\setfinitedifferenceoperatorsfourier)$ is obtained by taking the entry-wise Fourier transform of $\genericmatrix$.
              Moreover, we have that
              \begin{equation}\label{eq:TransformationCharPolyFourier}
                \charactpolynomial_{\genericmatrix} = \sum_{\indicepolynomials = 0}^{\velocitynumber} \coeffcharact_{\indicepolynomials} \polynomialunknown^{\indicepolynomials}, \qquad \xLongleftrightarrow[]{\fouriertransformletter} \qquad \charactpolynomial_{\hat{\genericmatrix}} = \sum_{\indicepolynomials = 0}^{\velocitynumber} \hat{\coeffcharact}_{\indicepolynomials} \polynomialunknown^{\indicepolynomials},
              \end{equation}
              where $(\coeffcharact_{\indicepolynomials})_{\indicepolynomials = 0}^{\indicepolynomials = \velocitynumber} \subset \setfinitedifferenceoperators$ and $(\fourier{\coeffcharact}_{\indicepolynomials})_{\indicepolynomials = 0}^{\indicepolynomials = \velocitynumber} \subset \setfinitedifferenceoperatorsfourier$.
            
              \subsection{Correspondence between the stability analysis for \fd and lattice Boltzmann schemes}

              Considering linear (or linearized) schemes written in the Fourier space is, thanks to the Parseval's identity, the standard setting to perform the $\elltwospace$ linear stability analysis both for \lbm and \fd schemes.
              Assume to deal only with one conserved variable, thus $\consmomentsnumber = 1$. 
            
              The polynomial associated with a linear \fd scheme -- or quite often, its Fourier representation -- is called amplification polynomial, see Chapter 4 of \cite{strikwerda2004finite}. The study of its roots in the Fourier space is the key of the so-called \emph{von Neumann} stability analysis.
              \begin{definition}[\emph{von Neumann} stability of a \fd scheme]\label{def:vonNeumannFD}
                Consider a multi-step linear \fd scheme for the variable $\fdvariable$ under the form\footnote{In this formulation, we do not account for the presence of source terms, since they do not play any role in the linear stability analysis.}
                \begin{equation}\label{eq:GenericFDScheme}
                  \sum_{\indicetimeshift = 0}^{\velocitynumber} \coeffamplificationpoly_{\velocitynumber - \indicetimeshift} \fdvariable^{\timevariable + 1 - \indicetimeshift} = 0, 
                \end{equation}
                for $(\coeffamplificationpoly_{\indicepolynomials})_{\indicepolynomials = 0}^{\indicepolynomials = \velocitynumber} \subset \setfinitedifferenceoperators$.
                Consider its amplification polynomial $\amplificationpolynomial \definitionequality \sum_{\indicepolynomials = 0}^{\indicepolynomials = \velocitynumber} \coeffamplificationpoly_{\indicepolynomials} \polynomialunknown^{\indicepolynomials}$, with corresponding amplification polynomial in the Fourier space $\fourier{\amplificationpolynomial} \definitionequality \sum_{\indicepolynomials = 0}^{\indicepolynomials = \velocitynumber} \fourier{\coeffamplificationpoly}_{\indicepolynomials} \polynomialunknown^{\indicepolynomials}$.
                We say that the \fd scheme \Cref{eq:GenericFDScheme} is stable in the \emph{von Neumann} sense if for every $\fourier{\rootamplificationpoly} :  [ -{\pi}/{\spacestep}, {\pi}/{\spacestep}  ]^{\spatialdimensionality} \to \complex$ such that $\fourier{\amplificationpolynomial} (\fourier{\rootamplificationpoly}(\vectorial{\freqvariable})) = \sum_{\indicepolynomials = 0}^{\indicepolynomials = \velocitynumber} \fourier{\coeffamplificationpoly}_{\indicepolynomials}(\vectorial{\freqvariable}) \fourier{\rootamplificationpoly}(\vectorial{\freqvariable})^{\indicepolynomials} = 0$, then
                \begin{enumerate}
                  \item $\lvert \fourier{\rootamplificationpoly}(\vectorial{\freqvariable}) \rvert \leq 1, $ for every $  \vectorial{\freqvariable} \in [ -{\pi}/{\spacestep}, {\pi}/{\spacestep} ]^{\spatialdimensionality}$.
                  \item If $\lvert \fourier{\rootamplificationpoly}(\vectorial{\freqvariable}) \rvert = 1$ for some $  \vectorial{\freqvariable} \in [ -{\pi}/{\spacestep}, {\pi}/{\spacestep} ]^{\spatialdimensionality}$, then $\fourier{\rootamplificationpoly}(\vectorial{\freqvariable})$ is a simple root.
                \end{enumerate}
              \end{definition}
              The conditions by \Cref{def:vonNeumannFD} are necessary and sufficient for stability (Theorem 4.2.1 in \cite{strikwerda2004finite}) if the scheme is explicitly independent of $\spacestep$ and $\timestep$. 
            
              Consider now the \lbm scheme \Cref{eq:SchemeAB} with linear (or linearized) equilibria, that is, there exists $\linearequilvector \in \reals^{\velocitynumber}$ such that $\vectorial{\momentletter}^{\atequilibrium} = \linearequilvector \momentletter_1 = (\linearequilvector \otimes \canonicalbasisvector_1) \vectorial{\momentletter}$.
              Writing the corresponding \fd scheme from \Cref{prop:ReductionFiniteDifferenceGeneral}, we have
              \begin{equation}
                \momentletter_1^{\timevariable + 1}  + \sum_{\indicetimeshift = 0}^{\velocitynumber-1} \coeffcharact_{\indicetimeshift}  \momentletter_1^{\timevariable + 1 - \velocitynumber + \indicetimeshift} - \left (\sum_{\indicetimeshift = 0}^{\velocitynumber - 1} \left ( \sum_{\ell=0}^{\indicetimeshift} \coeffcharact_{\velocitynumber + \ell - \indicetimeshift} \schememoments^{\ell} \schemeequil  \linearequilvector \otimes \canonicalbasisvector_1 \right )_{11} {\momentletter}_1^{\timevariable - \indicetimeshift} \right ) = 0,
              \end{equation}
              where $\charactpolynomial_{\schememoments} = \sum_{\indicepolynomials = 0}^{\indicepolynomials = \velocitynumber} \coeffcharact_{\indicepolynomials} \polynomialunknown^{\indicepolynomials}$.
              Rearranging gives
              \begin{equation}
                \momentletter_1^{\timevariable + 1}  + \sum_{\indicetimeshift = 0}^{\velocitynumber - 1}\left (  \coeffcharact_{\velocitynumber - 1 - \indicetimeshift}  - \left ( \sum_{\ell=0}^{\indicetimeshift} \coeffcharact_{\velocitynumber + \ell - \indicetimeshift} \schememoments^{\ell}  \schemeequil \linearequilvector \otimes \canonicalbasisvector_1 \right )_{11} \right ) \momentletter_1^{\timevariable - \indicetimeshift} = 0,
              \end{equation}
              which is a \fd scheme of the form given in \Cref{eq:GenericFDScheme} (with $\fdvariable = \momentletter_1$) by setting
              \begin{equation}\label{eq:DefinitionCoefficientsAmplification}
                \coeffamplificationpoly_{\indicepolynomials} = 
                \begin{cases}
                  1, \qquad &\text{if} \quad \indicepolynomials = \velocitynumber, \\
                  \coeffcharact_{\indicepolynomials}  -  \left ( \sum_{\ell=0}^{\ell = \velocitynumber - 1 - \indicepolynomials} \coeffcharact_{\indicepolynomials + 1 + \ell} \schememoments^{\ell}  \schemeequil \linearequilvector \otimes \canonicalbasisvector_1\right )_{11}, \qquad &\text{if} \quad \indicepolynomials \in \integerinterval{0}{\velocitynumber - 1}.
                \end{cases}
              \end{equation} 

              \begin{proposition}\label{prop:EquivalenceStability}
                Let $\consmomentsnumber = 1$ and consider the \lbm scheme \Cref{eq:SchemeAB} with linear equilibria, that is, there exists $\linearequilvector \in \reals^{\velocitynumber}$ such that $\vectorial{\momentletter}^{\atequilibrium} = \linearequilvector \momentletter_1 = (\linearequilvector \otimes \canonicalbasisvector_1) \vectorial{\momentletter}$.
                It thus reads $\vectorial{\momentletter}^{\timevariable + 1} = (\schememoments + \schemeequil \linearequilvector \otimes \canonicalbasisvector_1) \vectorial{\momentletter}^{\timevariable}$, where $\schememoments + \schemeequil \linearequilvector \otimes \canonicalbasisvector_1 \in \matrixspace_{\velocitynumber}(\setfinitedifferenceoperators)$.
                Then
                \begin{equation*}
                  \amplificationpolynomial \equiv \charactpolynomial_{\schememoments + \schemeequil \linearequilvector \otimes \canonicalbasisvector_1},
                \end{equation*}
                where $ \amplificationpolynomial \definitionequality \sum_{\indicepolynomials = 0}^{ \velocitynumber} \coeffamplificationpoly_{\indicepolynomials} \polynomialunknown^{\indicepolynomials}$, with $(\coeffamplificationpoly_{\indicepolynomials})_{\indicepolynomials = 0}^{\indicepolynomials = \velocitynumber}$ given by \Cref{eq:DefinitionCoefficientsAmplification}.
              \end{proposition}
              This result -- proved at the end of the section -- states that, under adequate assumptions, the amplification polynomial of the corresponding \fd scheme coincides with the characteristic polynomial associated with the original \lbm scheme.
              \Cref{prop:EquivalenceStability} has also confirmed that assuming the linearity of the equilibria and then performing the computation of the corresponding \fd scheme using the polynomial $\charactpolynomial_{\schememoments + \schemeequil \linearequilvector \otimes \canonicalbasisvector_1}$ yields the same result than performing the computation with $\charactpolynomial_{\schememoments}$ on the possibly non-linear scheme and then considering linear equilibria only at the very end.
              Thus, a similar notion of stability holds for \lbm schemes.
              \begin{definition}[\emph{von Neumann} stability of a \lbm scheme]\label{def:vonNeumannLBM}
                Let $\consmomentsnumber = 1$ and consider the \lbm scheme \Cref{eq:SchemeAB} with linear equilibria.
                It thus reads
                \begin{equation}\label{eq:LinearLBM}
                  \vectorial{\momentletter}^{\timevariable + 1} = (\schememoments + \schemeequil \linearequilvector \otimes \canonicalbasisvector_1) \vectorial{\momentletter}^{\timevariable},
                \end{equation}
                where $\schememoments + \schemeequil \linearequilvector \otimes \canonicalbasisvector_1 \in \matrixspace_{\velocitynumber}(\setfinitedifferenceoperators)$.
                We say that the \lbm scheme \Cref{eq:LinearLBM} is stable in the \emph{von Neumann} sense if for every $\vectorial{\freqvariable} \in  [ -{\pi}/{\spacestep}, {\pi}/{\spacestep} ]^{\spatialdimensionality}$, then every $\fourier{\rootamplificationpoly} \in \spectrum{\fourier{\schememoments}(\vectorial{\freqvariable}) + \fourier{\schemeequil}(\vectorial{\freqvariable}) \linearequilvector\otimes \canonicalbasisvector_1}$ is such that
                \begin{enumerate}
                  \item\label{item:firspoint} $ \lvert \fourier{\rootamplificationpoly} \rvert \leq 1$.
                  \item\label{item:secondpoint} If $ \lvert \fourier{\rootamplificationpoly} \rvert = 1$, then $\fourier{\rootamplificationpoly}$ is a simple eigenvalue of $\fourier{\schememoments}(\vectorial{\freqvariable}) + \fourier{\schemeequil}(\vectorial{\freqvariable}) \linearequilvector\otimes \canonicalbasisvector_1$.
                \end{enumerate}
                Here, $\spectrum{\cdot}$ denotes the spectrum of a matrix.
              \end{definition}
              \Cref{item:firspoint} alone, in \Cref{def:vonNeumannLBM}, coincides with the standard definition of stability for \lbm schemes, see \cite{sterling1996stability}.
              With \Cref{item:secondpoint}, we have been more precise on the subtle question of multiple eigenvalues\footnote{This question is not harmless since for instance the \scheme{1}{2} scheme rewrites as a leap-frog scheme \cite{dellacherie2014construction} if the relaxation parameter is equal to two (see Appendices). This very \fd scheme can suffer from linear growth of the solution due to this issue, see Chapter 4 of \cite{strikwerda2004finite}.} by bringing this definition closer to \Cref{def:vonNeumannFD}.
              Thus, \Cref{prop:EquivalenceStability} has the following Corollary.
              \begin{corollary}
                For $\consmomentsnumber = 1$, the \lbm scheme \Cref{eq:SchemeAB}, rewritten as \Cref{eq:LinearLBM} under linearity assumption on the equilibria, is stable in the \emph{von Neumann} sense according to \Cref{def:vonNeumannLBM} if and only if its corresponding \fd scheme obtained by \Cref{prop:ReductionFiniteDifferenceGeneral} is stable in the \emph{von Neumann} sense according to \Cref{def:vonNeumannFD}.
              \end{corollary}
              This result gives a precise and rigorous framework to the widely employed notion of stability \cite{sterling1996stability} for \lbm schemes.

              We finish on the proof of \Cref{prop:EquivalenceStability}.
              We need the following result concerning the determinant of matrices under rank-one updates, see \cite{ding2007eigenvalues} for the proof.
              \begin{lemma}[Matrix determinant]\label{lemma:MatrixDeterminant}
                Let $\genericcommutativering$ be a commutative ring, $\genericmatrix \in \matrixspace_{r}(\genericcommutativering)$ for some $r \in \naturals^{\star}$ and $\vectorial{u}, \vectorial{v} \in \genericcommutativering^r$, then $\determinant(\genericmatrix + \vectorial{u} \otimes \vectorial{v}) = \determinant(\genericmatrix) + \transpose{\vectorial{v}}\adjugate(\genericmatrix) \vectorial{u}$, where $\adjugate(\cdot)$ denotes the adjugate matrix, also known as classical adjoint.
              \end{lemma}
              We are ready to prove \Cref{prop:EquivalenceStability}.
              \begin{proof}
              Using \Cref{lemma:MatrixDeterminant}, one has
                \begin{align*}
                  \charactpolynomial_{\schememoments +  \schemeequil \linearequilvector \otimes \canonicalbasisvector_1}  :&= (-1)^{\velocitynumber}\determinant (\schememoments +  (\schemeequil\linearequilvector \otimes \canonicalbasisvector_1) - \polynomialunknown \matricial{\identity}), \\
                  &= (-1)^{\velocitynumber} \determinant (\schememoments  - \polynomialunknown \matricial{\identity}) + (-1)^{\velocitynumber} \transpose{\canonicalbasisvector_1} \adjugate(\schememoments - \polynomialunknown \matricial{\identity}) \schemeequil\linearequilvector, \\
                  &= \charactpolynomial_{\schememoments}+ (-1)^{\velocitynumber} \transpose{\canonicalbasisvector_1} \adjugate(\schememoments - \polynomialunknown \matricial{\identity}) \schemeequil\linearequilvector.
                \end{align*}
                By the definition of adjugate and by the Cayley-Hamilton \Cref{thm:CayleyHamilton}, we have
                \begin{align*}
                  (-1)^{\velocitynumber}(\schememoments - \polynomialunknown \matricial{\identity}) \adjugate(\schememoments - \polynomialunknown \matricial{\identity}) &= (-1)^{\velocitynumber}\determinant (\schememoments - \polynomialunknown \matricial{\identity}) \matricial{\identity} = (-1)^{\velocitynumber}\determinant (\schememoments - \polynomialunknown \matricial{\identity}) \matricial{\identity} \hspace{-0.08cm} - \hspace{-0.08cm}\charactpolynomial_{\schememoments} (\schememoments) \\
                  &= -\sum_{\indicepolynomials = 0}^{\velocitynumber} \coeffcharact_{\indicepolynomials} (\schememoments^{\indicepolynomials} - \polynomialunknown^{\indicepolynomials} \matricial{\identity}) = -\sum_{\indicepolynomials = 1}^{\velocitynumber} \coeffcharact_{\indicepolynomials} (\schememoments^{\indicepolynomials} - (\polynomialunknown \matricial{\identity})^{\indicepolynomials}), \\
                  &= -(\schememoments - \polynomialunknown \matricial{\identity}) \sum_{\indicepolynomials=1}^{\velocitynumber} \coeffcharact_{\indicepolynomials} \sum_{\ell = 0}^{\indicepolynomials - 1} \schememoments^{\ell} (\polynomialunknown \matricial{\identity})^{\indicepolynomials - 1 - \ell}, \\
                  &= -(\schememoments - \polynomialunknown \matricial{\identity}) \sum_{\indicepolynomials=1}^{\velocitynumber} \coeffcharact_{\indicepolynomials} \sum_{\ell = 0}^{\indicepolynomials - 1} \schememoments^{\ell} \polynomialunknown^{\indicepolynomials - 1 - \ell},
                \end{align*}
                where we have used that if $\genericmatrix, \genericmatrixtwo \in \matrixspace_{\velocitynumber}(\genericcommutativering)$ on a commutative ring, then $\genericmatrix^{\indicepolynomials} - \genericmatrixtwo^{\indicepolynomials} = (\genericmatrix - \genericmatrixtwo) (\genericmatrix^{\indicepolynomials-1} + \genericmatrix^{\indicepolynomials-2}\genericmatrixtwo + \dots + \genericmatrix\genericmatrixtwo^{\indicepolynomials - 2} + \genericmatrixtwo^{\indicepolynomials-1})$.
                We deduce that 
                \begin{equation}\label{eq:ExpressionOfAdjugate}
                  \adjugate(\schememoments - \polynomialunknown \matricial{\identity}) = - (-1)^{\velocitynumber}\sum_{\indicepolynomials=1}^{\velocitynumber} \coeffcharact_{\indicepolynomials} \sum_{\ell = 0}^{\indicepolynomials - 1} \schememoments^{\ell} \polynomialunknown^{\indicepolynomials - 1 - \ell}.
                \end{equation}
                This yields
                \begin{align*}
                  \charactpolynomial_{\schememoments +  \schemeequil \linearequilvector \otimes \canonicalbasisvector_1} (\polynomialunknown) = \polynomialunknown^{\velocitynumber} + \sum_{k = 0}^{\velocitynumber - 1} \coeffcharact_{\indicepolynomials} \polynomialunknown^{\indicepolynomials} - \transpose{\canonicalbasisvector_1}  \sum_{\indicepolynomials=1}^{\velocitynumber} \coeffcharact_{\indicepolynomials} \sum_{\ell = 0}^{\indicepolynomials - 1} \schememoments^{\ell} \polynomialunknown^{\indicepolynomials - 1 - \ell} \schemeequil\linearequilvector.
                \end{align*}
                Performing the following change of variable $t = \indicepolynomials - 1 - \ell \in \integerinterval{0}{\velocitynumber - 1}$ with $\ell \in \integerinterval{0}{\velocitynumber - 1 - t}$, thus $\indicepolynomials = t + 1 + \ell$, gives
                \begin{align*}
                  \charactpolynomial_{\schememoments +  \schemeequil \linearequilvector \otimes \canonicalbasisvector_1} (\polynomialunknown) &= \polynomialunknown^{\velocitynumber} + \sum_{\indicepolynomials = 0}^{\velocitynumber - 1} \left ( \coeffcharact_{\indicepolynomials}  - \transpose{\canonicalbasisvector_1} \sum_{\ell = 0}^{\velocitynumber - 1 - \indicepolynomials} \schememoments^{\ell} \schemeequil\linearequilvector \coeffcharact_{\indicepolynomials + 1 +\ell} \right ) \polynomialunknown^{\indicepolynomials}, \\
                  &= \polynomialunknown^{\velocitynumber} + \sum_{\indicepolynomials = 0}^{\velocitynumber - 1} \left ( \coeffcharact_{\indicepolynomials}  - \left (\sum_{\ell = 0}^{\velocitynumber - 1 - \indicepolynomials} \coeffcharact_{\indicepolynomials + 1 +\ell}  \schememoments^{\ell} \schemeequil\linearequilvector \otimes \canonicalbasisvector_1 \right )_{11} \right ) \polynomialunknown^{\indicepolynomials}.
                \end{align*}
                Thus we have that  $\amplificationpolynomial \definitionequality \sum_{\indicepolynomials = 0}^{\indicepolynomials = \velocitynumber} \coeffamplificationpoly_{\indicepolynomials} \polynomialunknown^{\indicepolynomials} =  \charactpolynomial_{\schememoments + \schemeequil \linearequilvector \otimes \canonicalbasisvector_1}$.
              \end{proof}
            
              \section{Convergence of \lbm schemes on an example}\label{sec:ConvergenceLBM}
              In this Section, we show on \Cref{ex:D1Q3OneConservedVariable} (taking $p = 1$ to simplify the stability analysis, see \Cref{ex:D1Q3OneConservedVariableP1}) that the theory available for multi-step \fd schemes can be used to study the underlying \lbm scheme.
              The target conservation law is the Cauchy problem
              \begin{equation}\label{eq:TargetPDE}
                \begin{cases}
                  \partial_t u(t, \spacevariable) + \latticevelocity \cfl \partial_x u (t, \spacevariable) = 0, \qquad &(t, \spacevariable) \in [0, T] \times \reals, \\
                  u(t=0, \spacevariable) = u_0(\spacevariable), \qquad &\spacevariable \in \reals.
                \end{cases}
              \end{equation}
              The equilibria are considered to be linear as in \Cref{sec:Stability}: $\momentletter_2^{\atequilibrium} = \latticevelocity \cfl \momentletter_1$ where $\cfl$ is the Courant number and $\momentletter_3^{\atequilibrium} = 2\latticevelocity^2 \foueriernumber \momentletter_1$ with $\foueriernumber$ the Fourier number.
              The corresponding \fd scheme from \Cref{ex:D1Q3OneConservedVariable} and \Cref{ex:D1Q3OneConservedVariableP1} is consistent with 
              \begin{equation}\label{eq:D1Q3NumericalTestEquivalentEquation}
                \partial_t \momentletter_1 + \latticevelocity \cfl \partial_x \momentletter_1 - \latticevelocity \spacestep \left (\frac{1}{\relaxparletter} - \frac{1}{2} \right ) \left (\frac{2}{3} (1 + \foueriernumber) - \cfl^2 \right ) \partial_{xx} \momentletter_1 = \bigO{\spacestep^2}.
              \end{equation}
              In what follows, we shall fix $\cfl = 1/2$.
              One can make the residual diffusion in this equation vanish if $\relaxparletter = 2$, which is a staple of \lbm schemes \cite{dubois2008equivalent, junk2008regular, graille2014approximation, simonis2020relaxation}, or by having $\foueriernumber = 3 \cfl^2 / 2 - 1$.
              We shall analyze both the case $\foueriernumber > 3 \cfl^2 / 2 - 1$, where expect only linear consistency with \Cref{eq:TargetPDE} or -- using the notations from \cite{strikwerda2004finite} -- where the scheme is accurate of order $[r, \rho] = [1, 2]$ and the case $\foueriernumber =  3 \cfl^2 / 2 - 1$, the scheme is second-order consistent with \Cref{eq:TargetPDE} or $[r, \rho] = [2, 3]$ accurate.
            
              \begin{figure}[h]
                \begin{center}
                    \includegraphics[width=.49\textwidth]{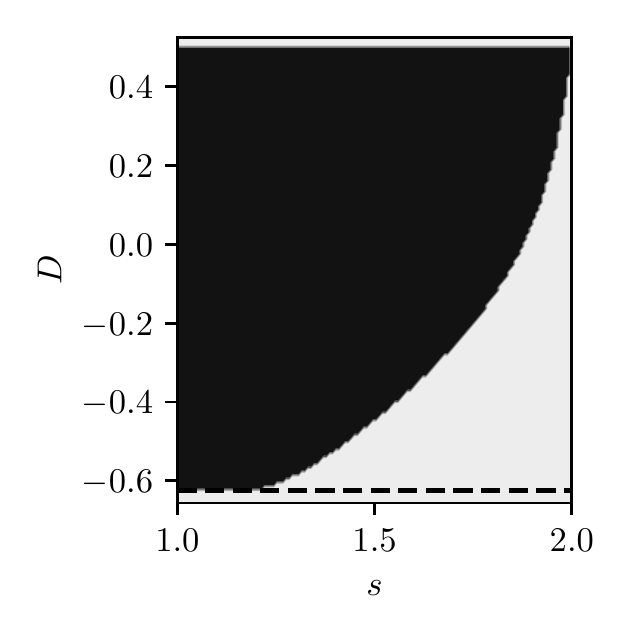}
                    \includegraphics[width=.49\textwidth]{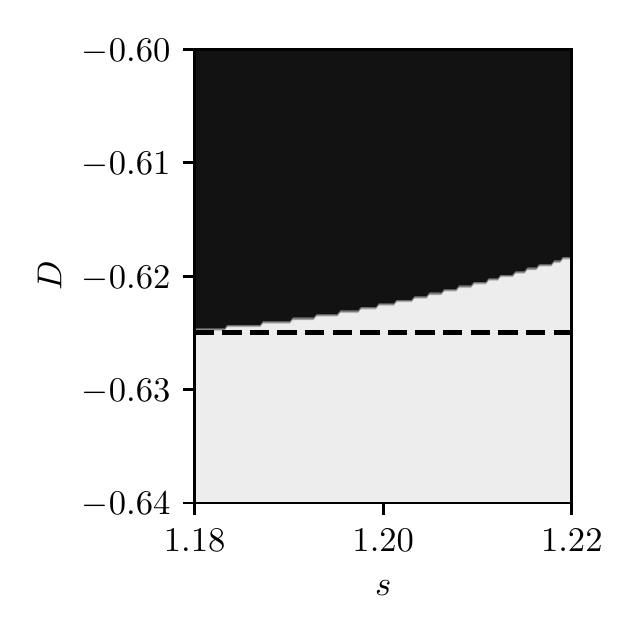}
                \end{center}\caption{\label{fig:stabilityregions}Stability region (in black), obtained numerically, as function of $\relaxparletter$ and $\foueriernumber$ for the \scheme{1}{3} of \Cref{ex:D1Q3OneConservedVariable}, considering $\cfl = 1/2$. The black dashed line corresponds to $\foueriernumber = 3\cfl^2 /2 - 1 = -0.625$, for which the residual diffusivity vanishes, see \Cref{eq:D1Q3NumericalTestEquivalentEquation}. The right image is a magnification of the left one close to $\relaxparletter = 1.2$.}
              \end{figure}
              The numerical \emph{von Neumann} stability analysis has been done and the result is shown in \Cref{fig:stabilityregions}. 
              One sees that enforcing positive residual diffusivity is necessary but not sufficient to obtain stability.
              Using the method from \cite{miller1971location} to locate the zeros of the amplification polynomial, we show the following.
              \begin{proposition}\label{prop:StabilityCharacterizationD1Q3}
                The amplification polynomial of the \fd scheme corresponding to the \scheme{1}{3} scheme from \Cref{ex:D1Q3OneConservedVariableP1} considered in this Section is a simple \emph{von Neumann} polynomial, namely fulfills \Cref{def:vonNeumannFD}, if the following constraints hold.
                  \begin{align*}
                    \frac{3}{2}\cfl^2 &- 1 \leq \foueriernumber \leq \frac{1}{2}, \qquad \text{and} \qquad
                    \max_{\coscomp \in [-1, 1]} \Biggl \{ \relaxparletter^2 \cfl^2 (1 + \coscomp) (1 + \Omega)^2 \\
                    &+\frac{4}{9} (2-\relaxparletter)(\foueriernumber + 1) (1-\Omega) \left ( (2-\relaxparletter)(\foueriernumber + 1) (1-\coscomp) (1 - \Omega) + 3 (\Omega^2- 1)\right )\Biggr \} \leq 0,
                  \end{align*}
                  where $\Omega = \Omega(\coscomp ; \foueriernumber, \relaxparletter) \definitionequality (1 - \relaxparletter) (\coscomp + 2 + 2\foueriernumber(1-\coscomp))/3$.
              \end{proposition}
              The first inequality from this Proposition gives only a necessary condition selecting a rectangle in the $(\relaxparletter, \foueriernumber)$ plane.
              The second one provides a sufficient condition yielding the non-straightforward profile visible on \Cref{fig:stabilityregions}.
              This comes from the fact that the maximum can be reached either on the boundary of $[-1, 1]$ (for $s \leq 1.18$ approximately) yielding the flat profile close to $s = 1$, or inside this compact (for $s > 1.18$), giving the tightening shape as $s$ increases towards $s = 2$.

              Using the generalization of Theorem 10.1.4 from \cite{strikwerda2004finite} to multi-step schemes for regular data and that of Corollary 10.3.2 for non-smooth data, one obtains the following convergence result for the \lbm scheme.
              \begin{proposition}[Convergence of the \scheme{1}{3} scheme]\label{prop:ConvergenceD1Q3}
                Consider the \scheme{1}{3} linear scheme of \Cref{ex:D1Q3OneConservedVariable} presented in this Section with a choice of $(\cfl, \foueriernumber, \relaxparletter)$ rendering a stable scheme according to \Cref{def:vonNeumannLBM}, as discussed in \Cref{prop:StabilityCharacterizationD1Q3}.
                The scheme is initialized with the point values of $u_0$ and at equilibrium.
                Then
                \begin{itemize}
                  \item For $\foueriernumber > 3 \cfl^2 / 2 - 1$, namely the corresponding \fd scheme is accurate of order $[r, \rho] = [1, 2]$.
                  \begin{itemize}
                    \item If $u_0 \in H^2$, the convergence of the \lbm scheme is linear:
                    \begin{equation*}
                        \lVert \evaluationOperator u(t^{\timevariable}, \cdot ) - \momentletter_1^{\timevariable} \rVert_{\ell^2, \spacestep} \leq C \spacestep \lVert u_0 \rVert_{H^2}, \qquad \timevariable \in \integerinterval{0}{[T/\timestep]},
                    \end{equation*}
                    where $\evaluationOperator$ is the evaluation operator such that $\evaluationOperator u : \lattice \to \reals$ with $(\evaluationOperator u)(\spacevariable) = u(\spacevariable)$ for every $\spacevariable \in \lattice$.
                    \item If $u_0 \in H^{\sigma}$ for any $\sigma < \sigma_0 < 2$ and there exists a constant $C(u_0)$ such that $\lVert u_0 \rVert_{H^{\sigma}} \leq C(u_0) / \sqrt{\sigma - \sigma_0}$, then
                    \begin{equation*}
                      \lVert \evaluationOperator u(t^{\timevariable}, \cdot ) - \momentletter_1^{\timevariable} \rVert_{\ell^2, \spacestep} \leq C \spacestep^{\sigma_0 / 2} \sqrt{\lvert \text{ln}(\spacestep)\rvert} C(u_0), \quad \timevariable \in \integerinterval{0}{[T/\timestep]}.\footnote{The logarithmic term is rarely observed in simulations.}
                  \end{equation*}
                  \end{itemize}
                  \item For $\foueriernumber = 3 \cfl^2 / 2 - 1$, namely the corresponding \fd scheme is accurate of order $[r, \rho] = [2, 3]$.
                  \begin{itemize}
                    \item If $u_0 \in H^3$, the convergence of the \lbm scheme is quadratic:
                    \begin{equation*}
                        \lVert \evaluationOperator u(t^{\timevariable}, \cdot ) - \momentletter_1^{\timevariable} \rVert_{\ell^2, \spacestep} \leq C \spacestep^2 \lVert u_0 \rVert_{H^3}, \qquad \timevariable \in \integerinterval{0}{[T/\timestep]}.
                    \end{equation*}
                    \item If $u_0 \in H^{\sigma}$ for any $\sigma < \sigma_0 < 3$ and there exists a constant $C(u_0)$ such that $\lVert u_0 \rVert_{H^{\sigma}} \leq C(u_0) / \sqrt{\sigma - \sigma_0}$, then
                    \begin{equation*}
                      \lVert \evaluationOperator u(t^{\timevariable}, \cdot ) - \momentletter_1^{\timevariable} \rVert_{\ell^2, \spacestep} \leq C \spacestep^{2 \sigma_0 / 3} \sqrt{\lvert \text{ln}(\spacestep) \rvert} C(u_0), \quad \timevariable \in \integerinterval{0}{[T/\timestep]}.
                  \end{equation*}
                  \end{itemize}
                \end{itemize}
                The constants $C$ have the following dependencies: $C = C(T, \cfl, \foueriernumber, \relaxparletter)$.
              \end{proposition}

              We now corroborate these results with numerical simulations, which are carried, for the sake of the numerical implementation, on the bounded domain $[-1, 1]$ enforcing periodic boundary conditions. 
              The final simulation time is $T = 1/2$ and $\latticevelocity = 1$.
              We stress the fact that we employ the \lbm scheme and not its corresponding \fd scheme.
              Guided by the considerations from \Cref{prop:ConvergenceD1Q3} in terms of regularity, we take different initial functions with various smoothness, inspired by \cite{strikwerda2004finite}.
              \begin{align*}
                  \text{(a)} \qquad u_0 (x) &= \chi_{\lvert x \rvert \leq 1/2}(x) \in H^{\sigma}, \quad \text{for any} \quad \sigma < \sigma_0 = 1/2.\\
                  \text{(b)} \qquad u_0 (x) &= (1-2\lvert x \rvert)\chi_{\lvert x \rvert \leq 1/2}(x)  \in H^{\sigma}, \quad \text{for any} \quad \sigma < \sigma_0 = 3/2.\\
                  \text{(c)} \qquad u_0 (x) &= \cos^2{(\pi x)}\chi_{\lvert x \rvert \leq 1/2}(x) \in H^{\sigma}, \quad \text{for any} \quad \sigma < \sigma_0 = 5/2.\\
                  \text{(d)} \qquad u_0 (x) &= \text{exp}\left (-1/{(1- \lvert 2x \rvert^2)} \right )\chi_{\lvert x \rvert \leq 1/2}(x) \in C_{c}^{\infty}.
              \end{align*}
              
              \begin{figure}[h]
                \begin{center}
                    \includegraphics[width=1.\textwidth]{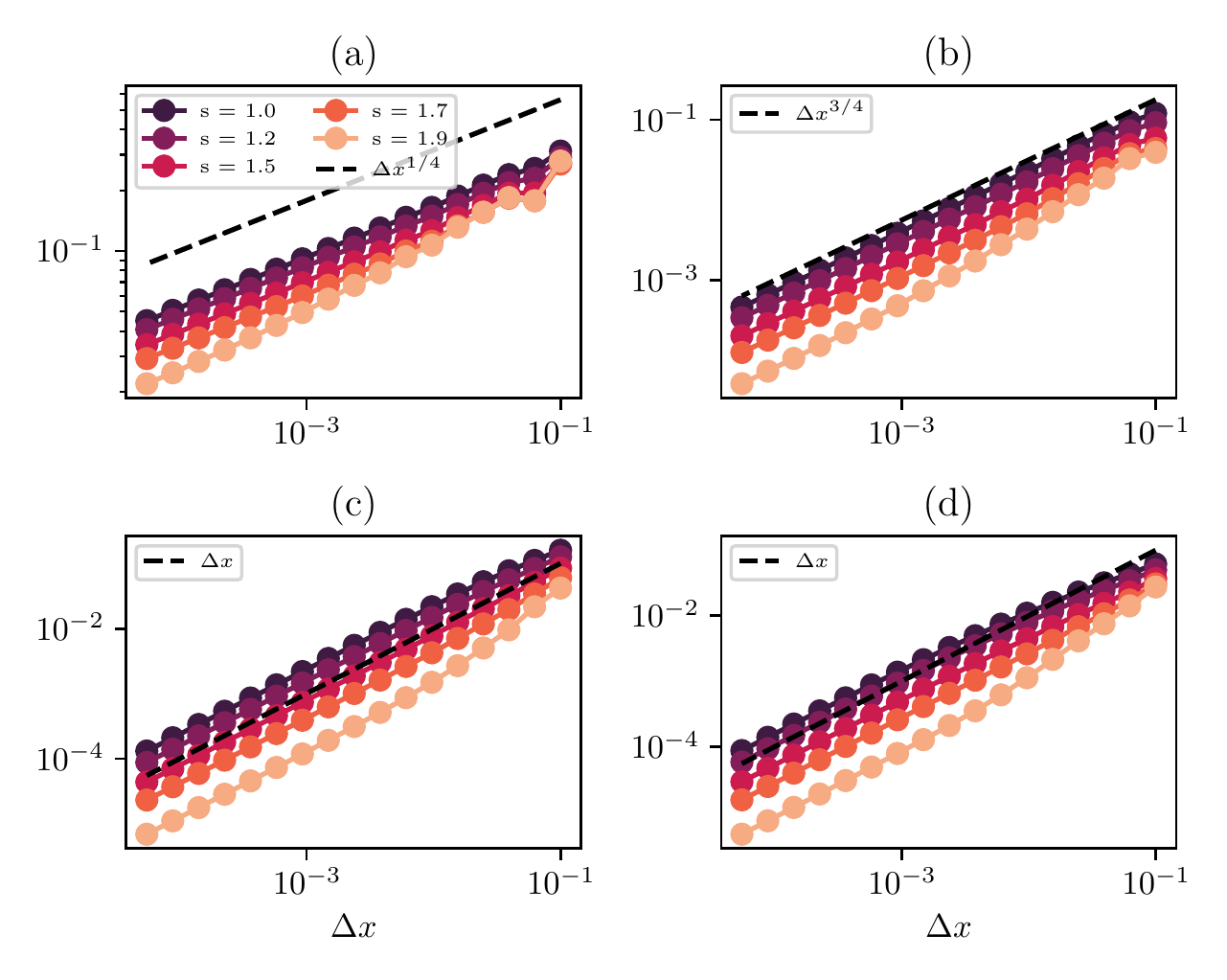}
                \end{center}\caption{\label{fig:convergence_I_order}$\foueriernumber = 0.4$. $\ell^2$ error at final time $T$ between the solution (conserved moment) of \lbm scheme and the exact solution, for different initial data (a), (b), (c) and (d) and different relaxation parameters $\relaxparletter$.}
              \end{figure}
              
              \begin{figure}[h]
                \begin{center}
                    \includegraphics[width=1.\textwidth]{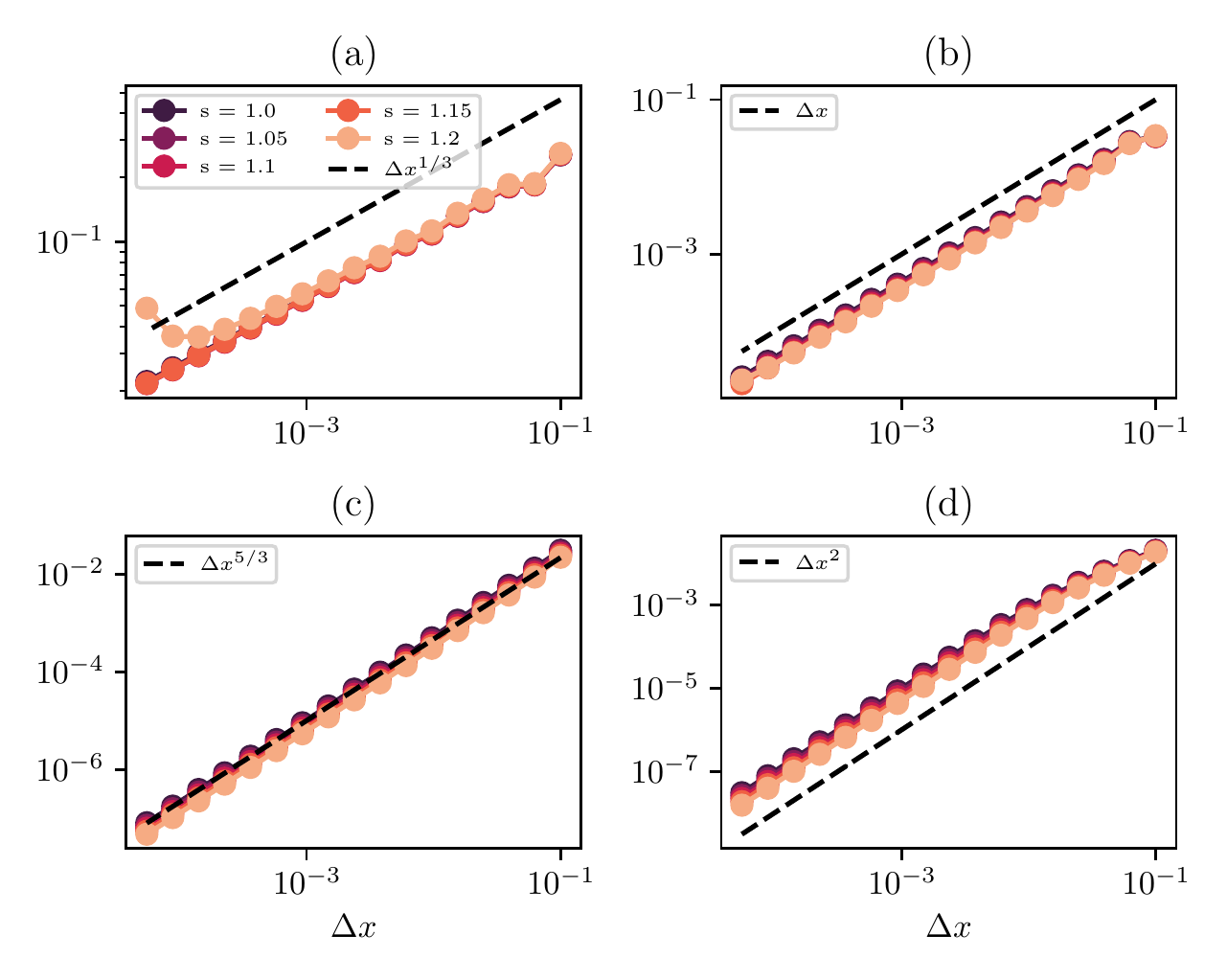}
                \end{center}\caption{\label{fig:convergence_II_order}$\foueriernumber = -0.625$. $\ell^2$ error at final time $T$ between the solution (conserved moment) of \lbm scheme and the exact solution, for different initial data (a), (b), (c) and (d) and different relaxation parameters $\relaxparletter$.}
              \end{figure}
              
              The numerical convergence for the case $\foueriernumber = 0.4$ is given on \Cref{fig:convergence_I_order}. 
              According to \Cref{fig:stabilityregions} and \Cref{prop:StabilityCharacterizationD1Q3}, we expect stability for every choice of $\relaxparletter$. Thus, the empirical convergence rates are in excellent agreement with \Cref{prop:ConvergenceD1Q3}. 
              The error constant is smaller for larger $\relaxparletter$, since for this choice, less numerical diffusion is present.
              
              Concerning the case $\foueriernumber = -0.625$ presented on \Cref{fig:convergence_II_order}, we had to utilize relaxation parameters $\relaxparletter$ close to one in order to remain in the stability region as prescribed by \Cref{fig:stabilityregions} and \Cref{prop:StabilityCharacterizationD1Q3}. As far as the scheme stays stable, for $\relaxparletter \leq 1.15$, we observe the expected convergence rates according to \Cref{prop:ConvergenceD1Q3}.
              Nevertheless, looking at the right image in \Cref{fig:stabilityregions}, we see that $\relaxparletter = 1.2$ is not in the stability region.
              This is why we observe, in (a) from \Cref{fig:convergence_II_order}, thus for the less smooth solution, that the scheme is not convergent.
              The instability originates from high-frequency modes which are abundant in the test case (a).
              This is the empirical evidence that the Lax-Richtmyer theorem \cite{lax1956survey} holds for \lbm schemes: an unstable scheme cannot be convergent.
            
            \section{Conclusions}\label{sec:Conclusions}
              In this paper, we proved that any \lbm scheme corresponds to a multi-step \fd scheme on the conserved moments, using a simple yet crucial result of linear algebra.
              This showed that \lbm schemes, in all their richness, fall in the framework of this latter category of well-known numerical schemes.
              Moreover, for linear problems and one conserved moment, we proved that the usual notion of stability employed for \lbm schemes is relevant, since it corresponds to the \emph{von Neumann} stability analysis for the \fd schemes.
              Therefore, the Lax-Richtmyer theorem \cite{lax1956survey,strikwerda2004finite}, stipulating that consistency and stability are the necessary and sufficient conditions for the convergence of linear \fd schemes, also holds for the \lbm schemes.
            
              A question left unanswered in this work, being the object of current investigations, concerns the link between the consistency for the corresponding \fd scheme and the theory of equivalent equations by \cite{dubois2008equivalent,dubois2019nonlinear}.
              In a complementary work \cite{bellotti2021equivalentequations}, we have proved that the two notions are equivalent up to second-order.
              The conjecture is that this holds for higher orders.
              The difficulty lies in the fact that performing \emph{a priori} Taylor expansions on the coefficients of the characteristic polynomial of $\schememoments$ is generally a hard task, due to their intrinsic non-linear dependence on $\schememoments$. Furthermore, the multi-step nature of the corresponding \fd scheme is an additional toil.

              \section*{Acknowledgments}
              The authors deeply thank L. Gouarin for the help in the implementation of the symbolic computations needed to check the provided examples.
              T. Bellotti friendly thanks his fellow PhD candidates C. Houpert, Y. Le Calvez, A. Louvet, M. Piquerez and D. Stantejsky for the useful discussions on algebra.
              This author is supported by a PhD funding (year 2019) from the Ecole polytechnique.

\bibliographystyle{acm}
\bibliography{biblio}

\section*{Appendices}
\subsection*{Proof of \Cref{prop:ReductionFiniteDifferenceAnnhilationFirstRowOnly}}
\begin{proof}
  By the choice of polynomial, we have that
  \begin{equation*}
    \left (\sum_{\indicepolynomials = 0}^{\degree(\minimalannpolynomial_{\schememoments})} \coeffannminimal_{\indicepolynomials} \schememoments^{\indicepolynomials} \right )_{1\cdot} = \left ( \coeffannminimal_0 + \sum_{\indicepolynomials = 1}^{\degree(\minimalannpolynomial_{\schememoments})} \coeffannminimal_{\indicepolynomials} (\schememoments^k)_{11}, 0, \dots, 0 \right ).
  \end{equation*}
  Restarting from the proof of \Cref{prop:ReductionFiniteDifferenceGeneral}, we have
  \begin{align*}
    \sum_{\indicetimeshift = 0}^{\degree(\minimalannpolynomial_{\schememoments})} \coeffannminimal_{\indicetimeshift} {\momentletter}_1^{\tilde{\timevariable} + \indicetimeshift} &= {\momentletter}_1^{\tilde{\timevariable} + \degree(\minimalannpolynomial_{\schememoments})} + \sum_{\indicetimeshift = 1}^{\degree(\minimalannpolynomial_{\schememoments}) - 1} \coeffannminimal_{\indicetimeshift} {\momentletter}_1^{\tilde{\timevariable} + \indicetimeshift} + \coeffannminimal_0 {\momentletter}_1^{\tilde{\timevariable}}, \\
    &= \left ( \left ( \sum_{\indicetimeshift = 0}^{\degree(\minimalannpolynomial_{\schememoments})} \coeffannminimal_{\indicetimeshift}  \schememoments^{\indicetimeshift} \right ) \vectorial{\momentletter}^{\tilde{\timevariable}} \right )_1 + \sum_{\indicetimeshift = 1}^{\degree(\minimalannpolynomial_{\schememoments})} \coeffannminimal_{\indicetimeshift}  \left ( \sum_{\ell = 0}^{\indicetimeshift-1} \schememoments^{\ell} \schemeequil  \vectorial{\momentletter}^{\atequilibrium}|^{\tilde{\timevariable} + \indicetimeshift - 1 - \ell} \right )_1, \\
    &=  \coeffannminimal_0 \momentletter^{\tilde{\timevariable}} + \left ( \sum_{\indicetimeshift = 1}^{\degree(\minimalannpolynomial_{\schememoments})} \coeffannminimal_{\indicetimeshift} (\schememoments^{\indicetimeshift})_{11} \right ) \momentletter_1^{\tilde{\timevariable}}  \\
    &+ \sum_{\indicetimeshift = 1}^{\degree(\minimalannpolynomial_{\schememoments})} \coeffannminimal_{\indicetimeshift}  \left ( \sum_{\ell = 0}^{\indicetimeshift-1} \schememoments^{\ell} \schemeequil  \vectorial{\momentletter}^{\atequilibrium}|^{\tilde{\timevariable} + \indicetimeshift - 1 - \ell} \right )_1,
  \end{align*}
  therefore
  \begin{align}
    {\momentletter}_1^{\tilde{\timevariable} + \degree(\minimalannpolynomial_{\schememoments})} = - \sum_{\indicetimeshift = 1}^{\degree(\minimalannpolynomial_{\schememoments}) - 1} \coeffannminimal_{\indicetimeshift} {\momentletter}_1^{\tilde{\timevariable} + \indicetimeshift} &+ \left ( \sum_{\indicetimeshift = 1}^{\degree(\minimalannpolynomial_{\schememoments})} \coeffannminimal_{\indicetimeshift} (\schememoments^{\indicetimeshift})_{11} \right ) \momentletter_1^{\tilde{\timevariable}}  \nonumber \\
    &+ \sum_{\indicetimeshift = 1}^{\degree(\minimalannpolynomial_{\schememoments})} \coeffannminimal_{\indicetimeshift}  \left ( \sum_{\ell = 0}^{\indicetimeshift-1} \schememoments^{\ell} \schemeequil  \vectorial{\momentletter}^{\atequilibrium}|^{\tilde{\timevariable} + \indicetimeshift - 1 - \ell} \right )_1. \label{eq:tmp3}
  \end{align}
  Performing the usual change of variable yields the result.
\end{proof}

\subsection*{Proof of \Cref{lemma:divisibility}}
  \begin{proof}
    The proof goes like the standard one of Lemma \ref{lemma:MinimalVsCharacteristic}.
    Consider $\minimalpolynomial_{\schememoments} = \polynomialunknown^{\degree(\minimalpolynomial_{\schememoments})} + \coeffminimal_{\degree(\minimalpolynomial_{\schememoments})-1}\polynomialunknown^{\degree(\minimalpolynomial_{\schememoments})-1} + \dots + \coeffminimal_1 \polynomialunknown + \coeffminimal_0$.
    Consider the Euclidian division between $\minimalpolynomial_{\schememoments}$ and $\minimalannrowpolynomial_{\schememoments}$: there exist $Q, R \in \setfinitedifferenceoperators[\polynomialunknown]$ such that
    \begin{equation*}
        \minimalpolynomial_{\schememoments} = \minimalannrowpolynomial_{\schememoments}Q + R, 
    \end{equation*}
    with either $0 < \degree(R) < \degree(\minimalannrowpolynomial_{\schememoments})$ or $\degree(R) =0$ (constant reminder polynomial).
    Let us indeed write
    \begin{align*}
        Q &= q_{\degree(\minimalpolynomial_{\schememoments}) - \degree(\minimalannrowpolynomial_{\schememoments})} \polynomialunknown^{\degree(\minimalpolynomial_{\schememoments}) - \degree(\minimalannrowpolynomial_{\schememoments})} + \dots + q_1 \polynomialunknown + q_0, \\
        R&= r_{\degree(R)} \polynomialunknown^{\degree(R)} + \dots + r_1 \polynomialunknown + r_0, \\
    \end{align*}
    Suppose that $R \not \equiv 0$, then we have for every $\indicescolumns \in \integerinterval{1}{\velocitynumber}$
    \begin{gather*}
        \overbrace{(\schememoments^{\degree(\minimalpolynomial_{\schememoments})})_{1\indicescolumns} + \coeffminimal_{\degree(\minimalpolynomial_{\schememoments})-1}(\schememoments^{\degree(\minimalpolynomial_{\schememoments})-1})_{1\indicescolumns} + \dots + \coeffminimal_1 (\schememoments)_{1\indicescolumns} + \coeffminimal_0 \delta_{1\indicescolumns}}^{=0} \\
        =  r_{\degree(R)} (\schememoments^{\degree(R})_{1\indicescolumns} + \dots + r_1 (\schememoments)_{1\indicescolumns} + r_0\delta_{1\indicescolumns} + \\
        \underbrace{\left ( (\schememoments^{\degree(\minimalannrowpolynomial_{\schememoments})})_{1\indicescolumns} + \coeffannminimal_{\degree(\minimalannrowpolynomial_{\schememoments})-1} (\schememoments^{\degree(\minimalannrowpolynomial_{\schememoments})-1})_{1\indicescolumns} + \dots + \coeffannminimal_1 (\schememoments)_{1\indicescolumns} + \coeffannminimal_0 \delta_{1\indicescolumns}\right )}_{= 0} \\
        \times \left ( q_{\degree(\minimalpolynomial_{\schememoments}) - \degree(\minimalannpolynomial_{\schememoments})} (\schememoments^{\degree(\minimalpolynomial_{\schememoments}) - \degree(\minimalannpolynomial_{\schememoments})})_{1\indicescolumns} + \dots + q_1 (\schememoments)_{1\indicescolumns}+ q_0 \delta_{1\indicescolumns}\right ),
    \end{gather*}
    thus
    \begin{equation*}
     r_{\degree(R)} (\schememoments^{\degree(R})_{1\indicescolumns} + \dots + r_1 (\schememoments)_{1\indicescolumns} + r_0 \delta_{1\indicescolumns} = 0, \qquad \indicescolumns \in \integerinterval{1}{\velocitynumber},
    \end{equation*}
    with $0 <\degree(R) < \degree(\minimalannrowpolynomial_{\schememoments})$, which contradicts the minimality of $\minimalannrowpolynomial_{\schememoments}$. Thus necessarily $\degree(R) = 0$ so the polynomial is constant, but to have the previous property, the constant must be zero, thus $R \equiv 0$.
  \end{proof}

\subsection*{Additional examples}

In this section, we gather more examples concerning the application of our theory to \lbm schemes which can be found in the literature.

\subsubsection*{\scheme{1}{2} with one conservation law}

Consider the scheme by \cite{dellacherie2014construction, graille2014approximation} taking $\spatialdimensionality = 1$ and $\velocitynumber = 2$ with $\normalizedvelocityletter_1 = 1$ and $\normalizedvelocityletter_2 = -1$ and
\begin{equation}\label{eq:D1Q2OneConservedMoment}
  \momentsmatrix = \left ( \begin{matrix}
                           1 & 1 \\
                           \latticevelocity & - \latticevelocity
                          \end{matrix} \right ), \qquad
  \relaxationmatrix = \diagmatrix(0, s), \quad \text{with} \quad s \neq 1.
\end{equation}
The scheme can be used to simulate a non-linear scalar conservation law (advection, Burgers, \emph{etc.}) using an acoustic scaling and a non-linear diffusion equation with a parabolic scaling.
However, the scheme is not rich enough to simulate more complex equations.
As already pointed out in the introduction , the \fd equivalent of this scheme has already been studied by \cite{dellacherie2014construction} in the case where the equilibria are linear functions.

It can be easily seen, even by hand since dealing with a $2\times 2$ matrix, that
\begin{equation*}
  \charactpolynomial_{\schememoments} = \polynomialunknown^2 - \frac{1}{2}(2-s)(\basicx + \conj{\basicx}) \polynomialunknown + (1-s).
\end{equation*}
The minimal polynomial coincides with the characteristic polynomial. This can be seen, as usual, by trying to consider $\alpha_0$ and $\alpha_1$ such that
\begin{equation*}
  \alpha_0 \matricial{\identity} + \alpha_1 \schememoments = 
  \begin{pmatrix}
      \alpha_0 + \frac{(\basicx + \conj{\basicx})}{2}\alpha_1 & \frac{(1-s)(\basicx - \conj{\basicx})}{2\latticevelocity}\alpha_1 \\ 
      \frac{\latticevelocity(\basicx - \conj{\basicx})}{2}\alpha_1 & \alpha_0 + \frac{(1-s)(\basicx + \conj{\basicx})}{2}\alpha_1
  \end{pmatrix} = 
  \begin{pmatrix}
   0 & 0 \\
   0 & 0
  \end{pmatrix}.
\end{equation*}
The only way of annihilating the first entry is to take $\alpha_0 = 0$, which is trivial.
Thus the minimal polynomial is of degree $2$ and then coincides with the characteristic polynomial.
The equivalent \fd scheme is 
  \begin{equation*}
      \momentletter_1^{\timevariable + 1} = \frac{1}{2} (2-s)(\basicx + \conj{\basicx})\momentletter_1^{\timevariable} - (1-s)\momentletter_1^{\timevariable - 1} +  \frac{s(\basicx - \conj{\basicx})}{2\latticevelocity} \momentletter_2^{\atequilibrium}|^{\timevariable}.
  \end{equation*}
  The scheme is a $\theta$-scheme between a Lax-Friedrichs scheme (for $s = 1$) and a leap-frog scheme (for $s = 2$).

\subsubsection*{\scheme{1}{3} SRT for one conservation law}

  Consider the \scheme{1}{3} SRT scheme by \cite{fuvcik2021equivalent}, also corresponding to that of \cite{suga2010accurate} which reads with our notations $\spatialdimensionality = 1$, $\velocitynumber = 3$ and $\normalizedvelocityletter_1 = 0$, $\normalizedvelocityletter_2 = 1$ and $\normalizedvelocityletter_3 = -1$ and 
  \begin{equation*}
    \momentsmatrix = 
    \begin{pmatrix}
      1 & 1 & 1 \\
      0 & \latticevelocity & -\latticevelocity \\
      0 & \latticevelocity^2 & \latticevelocity^2
    \end{pmatrix}, \qquad
    \relaxationmatrix = \diagmatrix(0, \omega, \omega), \quad \text{with} \quad \omega \neq 1,
  \end{equation*}
  The characteristic polynomial, corresponding to the minimal polynomial is
  \begin{equation*} 
    \charactpolynomial_{\schememoments} = \polynomialunknown^3 + (\omega(\basicx + \conj{\basicx}) - (\basicx + 1 +\conj{\basicx})) \polynomialunknown^2 + (1-\omega) ((\basicx + \conj{\basicx}) + (1-\omega)) \polynomialunknown - (1-\omega)^2.
  \end{equation*}
  Hence the equivalent \fd scheme is 
  \begin{align*}
    \momentletter_1^{\timevariable + 1} &= (1-\omega)(\basicx + \conj{\basicx}) \momentletter_1^{\timevariable} + \momentletter_1^{\timevariable} -  (1-\omega)(\basicx + \conj{\basicx}) \momentletter_1^{\timevariable - 1} - (1-\omega)^2 \momentletter_1^{\timevariable - 1} \\
    &+ (1-\omega)^2 \momentletter_1^{\timevariable - 2} + \frac{\omega(\basicx - \conj{\basicx})}{2\latticevelocity}\momentletter_2^{\atequilibrium}|^{\timevariable} - \frac{\omega(1-\omega)(\basicx - \conj{\basicx})}{2\latticevelocity} \momentletter_2^{\atequilibrium}|^{\timevariable - 1} \\
    &+ \frac{\omega(\basicx - 2 + \conj{\basicx})}{2\latticevelocity^2} \momentletter_3^{\atequilibrium}|^{\timevariable} + \frac{\omega(1-\omega) (\basicx - 2 + \conj{\basicx})}{2\latticevelocity^2} \momentletter_3^{\atequilibrium}|^{\timevariable - 1},
  \end{align*}
  coinciding with the one found by \cite{fuvcik2021equivalent}.

\subsubsection*{\scheme{1}{3} MRT for one conservation law}

Consider the \scheme{1}{3} MRT scheme by \cite{fuvcik2021equivalent}, which is constructed in the same way than the previous one except for $\relaxationmatrix = \diagmatrix(0, \omega_2, \omega_3)$ with $\omega_2, \omega_3 \neq 1$.
The characteristic and minimal polynomial coincide and are given by
\begin{align*}
  \charactpolynomial_{\schememoments}= \polynomialunknown^3 &+ (-1 + (\basicx + \conj{\basicx})(\omega_2/2 + \omega_3/2 - 1))\polynomialunknown^2 \\
  &+ (1 + \omega_2 \omega_3 - \omega_2 - \omega_3 + (1 - \omega_2/2 - \omega_3/2)(\basicx + \conj{\basicx})) \polynomialunknown \\
  &- (1-\omega_2)(1-\omega_3).
\end{align*}
Then the equivalent \fd scheme is 
\begin{align*}
  \momentletter_1^{\timevariable + 1} &= (1-\omega_2/2 - \omega_3/2)(\basicx + \conj{\basicx}) \momentletter_1^{\timevariable} + \momentletter_1^{\timevariable} -  (1-\omega_2/2 - \omega_3/2)(\basicx + \conj{\basicx}) \momentletter_1^{\timevariable - 1} \\
  &- (1-\omega_2 - \omega_3 + \omega_2 \omega_3) \momentletter_1^{\timevariable - 1} + (1-\omega_2)(1-\omega_3) \momentletter_1^{\timevariable - 2} \\
  &+ \frac{\omega_2(\basicx - \conj{\basicx})}{2\latticevelocity}\momentletter_2^{\atequilibrium}|^{\timevariable} - \frac{\omega_2(1-\omega_3)(\basicx - \conj{\basicx})}{2\latticevelocity} \momentletter_2^{\atequilibrium}|^{\timevariable - 1} \\
  &+ \frac{\omega_3(\basicx - 2 + \conj{\basicx})}{2\latticevelocity^2} \momentletter_3^{\atequilibrium}|^{\timevariable} + \frac{\omega_3(1-\omega_2) (\basicx - 2 + \conj{\basicx})}{2\latticevelocity^2} \momentletter_3^{\atequilibrium}|^{\timevariable - 1},
\end{align*}
corresponding to the one found by \cite{fuvcik2021equivalent}.

\subsection*{\scheme{2}{4} for one conservation law}

  Consider $\spatialdimensionality = 2$ and $\velocitynumber = 4$ with $\vectorial{\normalizedvelocityletter}_1 = \transpose{(1, 0)}$, $\vectorial{\normalizedvelocityletter}_2 = \transpose{(0, 1)}$, $\vectorial{\normalizedvelocityletter}_3 = \transpose{(-1, 0)}$ and $\vectorial{\normalizedvelocityletter}_4 = \transpose{(0, -1)}$ and
  \begin{equation}
      \momentsmatrix = 
      \left (
      \begin{matrix}
       1 & 1 & 1 & 1\\
       \latticevelocity & 0 & -\latticevelocity & 0 \\
       0 & \latticevelocity & 0 & -\latticevelocity \\
       \latticevelocity^2 & -\latticevelocity^2 & \latticevelocity^2 & -\latticevelocity^2
      \end{matrix}
      \right ), \qquad
      \relaxationmatrix = \diagmatrix(0, s, s, 1),
      \quad \text{with} \quad s \neq 1.
  \end{equation}
  Therefore $\consmomentsnumber = 1$ and $\nontrivialnumber = 2$.
  This can be used, for example, coupled with other schemes of the same nature (building what we call a ``vectorial scheme'' \cite{dubois2014simulation}) to easily simulate systems of non-linear conservation laws for $\spatialdimensionality = 2$, see \cite{bellotti2021multidimensional}. 
  After some computation, the characteristic polynomial of $\schememoments$ reads
  \begin{align*}
      \charactpolynomial_{\schememoments} = \hspace{-0.5mm} \polynomialunknown^3 &+ (2s-3) \frac{(\basicx + \conj{\basicx} + \basicy + \conj{\basicy})}{4} \polynomialunknown^2 + (1-s)\left ( (2-s)\frac{(\basicx \basicy + \conj{\basicx}\basicy+\basicx\conj{\basicy}+\conj{\basicx}\conj{\basicy})}{4} + 1 \right )\polynomialunknown \\
      &- (1-s)^2  \frac{(\basicx + \conj{\basicx} + \basicy + \conj{\basicy})}{4}.
  \end{align*}
  One can check as usual that it coincides with the minimal polynomial.
  The equivalent \fd scheme taking $\momentletter_4^{\atequilibrium} \equiv 0$ for simplicity is
  \begin{align*}
    \momentletter_1^{\timevariable+1} = &-(2s-3)\averageaxis \momentletter_1^{\timevariable} - (1-s) \momentletter_1^{\timevariable-1} - (1-s)(2-s)\averagediagonal \momentletter_1^{\timevariable-1} +(1-s)^2 \averageaxis \momentletter_1^{\timevariable-2} \\
    &+\frac{s}{2\latticevelocity}(\basicx - \conj{\basicx})\momentletter_2^{\atequilibrium}|^{\timevariable} + \frac{s}{2\latticevelocity}(\basicy - \conj{\basicy})\momentletter_3^{\atequilibrium}|^{\timevariable} \\
    &-\frac{s(1-s)}{\latticevelocity}\frac{1}{2}\left (\basicy \frac{(\basicx - \conj{\basicx})}{2} + \conj{\basicy} \frac{(\basicx - \conj{\basicx})}{2}\right ) \momentletter_2^{\atequilibrium}|^{\timevariable-1} \\
    &-\frac{s(1-s)}{\latticevelocity}\frac{1}{2}\left (\basicx \frac{(\basicy - \conj{\basicy})}{2} + \conj{\basicx} \frac{(\basicy - \conj{\basicy})}{2}\right ) \momentletter_3^{\atequilibrium}|^{\timevariable-1},
  \end{align*}
  where we have introduced the short-hands $\averageaxis \definitionequality (\basicx + \conj{\basicx} + \basicy + \conj{\basicy})/4 \in \setfinitedifferenceoperators$ and $\averagediagonal \definitionequality (\basicx \basicy + \basicx \conj{\basicy} + \conj{\basicx} \basicy + \conj{\basicx} \conj{\basicy})/4 \in \setfinitedifferenceoperators$, yielding respectively the average between neighbors along the axis and along the diagonals.

\end{document}